\documentclass{article}

\usepackage{arxiv}

\usepackage[utf8]{inputenc} 
\usepackage[T1]{fontenc}    
\usepackage{hyperref}       
\usepackage{url}            
\usepackage{booktabs}       
\usepackage{amsfonts}       
\usepackage{nicefrac}       
\usepackage{microtype}      

\usepackage{amsmath,amsthm,amsfonts,mathrsfs,anysize,fancyhdr,epsfig,mdframed, float,graphicx,dsfont}

\usepackage{chngcntr}
\usepackage{algorithm,algorithmic}
\theoremstyle{plain}
\newtheorem{theorem}{Theorem}

\newtheorem{lemma}[theorem]{Lemma}                              
\newtheorem{proposition}[theorem]{Proposition}
\newtheorem{corollary}[theorem]{Corollary}
\usepackage{enumitem}
\theoremstyle{definition}

\newtheorem{remark}[theorem]{Remark}
\newtheorem{assumption}[theorem]{Assumption}

\usepackage{chngcntr}
\usepackage{xcolor}
\definecolor{pao_green}{rgb}{0.0, 0.5, 0.1}

\usepackage{mathtools}                                                      
\mathtoolsset{showonlyrefs=true}                                    

\title{On stochastic mirror descent with interacting particles: convergence properties and variance reduction}
\author{A. Borovykh\thanks{Department of Computing, Imperial College London, \textbf{e-mail: }a.borovykh@imperial.ac.uk}, N. Kantas\thanks{Department of Mathematics, Imperial College London, \textbf{e-mail: }n.kantas@imperial.ac.uk},  P. Parpas\thanks{Department of Computing, Imperial College London, \textbf{e-mail: }panos.parpas@imperial.ac.uk}, G. A. Pavliotis\thanks{Department of Mathematics, Imperial College London, \textbf{e-mail:} g.pavliotis@imperial.ac.uk}}

\begin{document}

\maketitle 

\begin{abstract}
An open problem in optimization with noisy information is the computation of an exact minimizer that is independent of the amount of noise. 
A standard practice in stochastic approximation algorithms is to use a decreasing step-size. This however leads to a slower convergence. A second alternative is to use a fixed step-size and run independent replicas of the algorithm and average these. A third option is to run replicas of the algorithm and allow them to interact. It is unclear which of these options works best. To address this question, we reduce the problem of the computation of an exact minimizer with noisy gradient information to the study of stochastic mirror descent with interacting particles. We study the convergence of stochastic mirror descent and make explicit the tradeoffs between communication and variance reduction. We provide theoretical and numerical evidence to suggest that interaction helps to improve convergence and reduce the variance of the estimate. 
\end{abstract}

\section{Introduction}\label{sec:intro}
Optimization models that arise in artificial intelligence and statistical learning applications often include noisy estimates of the function and its gradient. This is the case when for example the gradient is computed over a subset (or mini-batch) of the data. In such a situation it is known that the optimization algorithm will converge to a neighborhood of the minimizer \cite{gower2019sgd}. The size of the neighborhood depends on the amount of noise. In addition, for constrained optimization problems noise can violate the constraints making the situation even more complex. 

In various applications it can be beneficial to be able to control the fluctuations around the true minimum. The conventional way to control the error is to decrease the step size. Theoretical analysis suggests step sizes which are slow in practice, e.g. $O(1/t)$ in \cite{mertikopoulos18}. An alternative is to use a vanishing noise variance \cite{mertikopoulos18a} or heuristics such as to increase the batch size over time (see e.g. \cite{smith17}); this however increases the computational costs and is difficult to tune. Another option is to run independent replicas of the algorithm. We will refer to each of these runs as a particle. The question we address in this paper is whether it is beneficial to allow these particles to interact with each other. We study this question using the general framework of Stochastic Mirror Descent (SMD) (see \cite{nemirovsky83}). SMD can be used to solve constrained and unconstrained problems, and is known to be an optimal algorithm for certain classes of optimization problems (\cite{beck17book}). 

In this paper we will consider generic convex optimization problems of the form,
\begin{align}
\min_{x\in\mathcal{X}}\{f(x)\},
\end{align}
where  $\mathcal{X}\subset\mathbb{R}^d$ is a closed convex set that describes the constraints. We are interested in investigating the performance and properties when the minimizer $x^*$ is estimated using the following  It\^o stochastic differential equation (SDE)
\begin{align}\label{eq:interact}
dz_t^i = -\nabla f(x^i_t) dt + \sum_{j=1}^N A_{ij}(z_t^j-z_t^i) dt + \sigma dB_t^i,  \quad x^i_t=\nabla\Phi^*(z_t^i),\quad i=1,...,N,
\end{align}
where each particle is driven by independent Brownian motions $B_t^i$ and $\Phi$ is the mirror map used in Mirror Descent (MD); we will present more details in Section \ref{sec:prelim}. The interesting feature here is that particles interact through the matrix $A=\{A_{ij}\}_{i,j=1}^N$, which is a $N\times N$ doubly-stochastic matrix representing the interaction weights. This interaction will attract  particles towards each other. The matrix $A$ represents an interaction graph which imposes communication constraints on the agents: each particle $i$ can communicate directly only with its immediate neighbors, i.e. $j\in\{1,\dots,N\}$ for whom $A_{ij}\neq 0$. In the absence of interactions (i.e. when $A=0$) the dynamics would correspond to independent replicas of SMD. 

The aim of this paper is to demonstrate the advantages of using the interacting particle system, in comparison to independent copies of SMD. We will show that interaction reduces the variance of the estimation of $x^*$ and can improve the convergence properties including the convergence rate to the stationary distribution. In particular, we show that for a (strongly) convex objective the distance to the optimum is bounded by a term related to the standard optimization error in which the noise variance is reduced by a factor of $N$ and an interaction term measuring the deviation of each particle from the system average. This latter term is bounded, so that under certain assumptions consensus and consequently convergence to the optimum is achieved. Using logarithmic Sobolev inequalities we can furthermore improve on this rate and show that the convergence to a stationary distribution can be achieved at an exponential rate. In addition, we believe that the ability to impose communication constraints through the interaction $A$ will be beneficial also from a practical point of view as it will decrease communication costs in a parallel implementation. In this paper we will work with the continuous time formulation for SMD and its proposed interacting version. The main reason behind this is that the analysis for continuous time dynamics provides a clear and complete picture for the benefit of interaction. 
In our numerical experiments we used a simple forward Euler discretization scheme and we expect that the results can only improve with more sophisticated schemes. A detailed error analysis of the discretization scheme is beyond the scope of this paper.


\subsection{Related Literature}

It is a well-known problem that stochastic optimization algorithms converge to a neighborhood of the (local) optimal solution. The size of this neighborhood is proportional to the noise variance (or the second moment of the sample gradient). Traditionally, common strategies for mitigating this include using a decreasing stepsize or attaining vanishing noise variance by increasing the batch size. Both of these come at additional computational costs. In the context of SMD, the authors of \cite{mertikopoulos18} study the convergence of SMD under the assumption that either the variance of the noise decreases over time or the step size is reduced exponentially slow. Various strategies have been also proposed to mitigate the effect of this noise variance to enable convergence closer to the optimum. The effect of sampling strategies such as importance sampling \cite{needell2014stochastic} can decrease the effect of the noise variance. Various variance reduction methods have been proposed, e.g. \cite{johnson2013accelerating}, \cite{defazio2014saga}, \cite{gorbunov2019unified}. The work of \cite{krichene17} and \cite{gu18} studied an accelerated version of SMD, however the distance to the optimum remains a function of the noise.

In this work we propose to use interacting particles to achieve a similar variance reduction. The work most closely related to ours is that of \cite{raginsky12} in continuous time and \cite{duchi11} in discrete time. The authors in \cite{raginsky12} study continuous and interacting SMD however only attained convergence rates for linear dynamics in the mirror domain and did not present any numerical results. The work of \cite{duchi11} covers the deterministic and noisy gradient setting in discrete time in a distributed setting; our work is in continuous time with all particles optimizing the same objective and uses different assumptions on the noise dynamics. 

Our work also has parallels with the vast distributed optimization literature from which we list some indicative recent references: \cite{lin2016distributed}, \cite{shahrampour2017distributed}, \cite{koloskova2019decentralized}, \cite{seaman2017optimal}. We remark that our objective is different from that in distributed optimization. Our goal here is to study how interaction can improve the convergence properties of SMD and analyze the variance reduction effects, while in distributed optimization the task is to optimize efficiently an objective function that is distributed across different nodes. For example a distributed version of the interacting SDE above would require using a different drift $-\nabla f_i(x_t^t)$ in the dynamics. There are also parallels with the work on consensus and synchronization, see e.g. \cite{belykh2004connection}, \cite{shi2015network}, \cite{yu2011consensus}, \cite{yu2008local} and \cite{yu2008global}. Our work generalizes some of these results to the SMD setting and combine it with results on convergence to the minimum. There is also parallel work using stochastic gradient descent (SGD) for constrained sampling problems considered in \cite{hsieh18}, \cite{fox18}, \cite{yokoi19}. While there are many similarities, the goals of sampling and optimization are different and different dynamics are used in each case \cite{ma2019sampling}; specifically, in optimization one wants to converge to the optimizer, while in sampling the objective is to converge to the correct invariant measure. 

\subsection{Contributions and Organization} 

Our main focus is on the convergence properties of stochastic mirror descent with interacting particles (ISMD). We derive regret bounds for the proximity of $f(x^i_t)$ to the optimal value using Lyapunov-based arguments. The application of Lyapunov techniques to optimization problems is an established approach. Our results have similarities with \cite{raginsky12} and \cite{mertikopoulos18} bar a number of differences. We use different Lyapunov functions and analyze a general convex cost function. We show that in the case of ISMD there is a tradeoff between communication cost, i.e. how many particles interact, and variance reduction. In particular, with a fixed learning rate and non-vanishing noise variance, interaction between the particles can reduce the size of the neighborhood around the optimum to which the algorithm converges.  Furthermore, we show that the particles converge to an area around the optimizer at an exponential rate using log-Sobolev inequalities and Bakry-Emery theory (see \cite{bakry1997sobolev} and \cite{bakry2013analysis}). Although such results are standard for the analysis of interacting SDEs with convex potentianls (e.g. \cite{malrieu01,veretennikov2006ergodic}), they have not been considered previously in the context of SMD. In addition, in our setting \emph{the particles are not restricted to all interact simultaneously}, so our results deviate from the mean field type analysis like in \cite{malrieu01,veretennikov2006ergodic}. 

Our contributions can be summarized as follows:
\begin{itemize}
\item We propose interaction between $N$ particles as a way of controlling the convergence of a stochastic optimization algorithm -- and in particular the distance to the minimum. In the presence of noise in the gradients interacting particles are an effective alternative to vanishing learning rates or noise variances. 
\item We establish that interaction leads to variance reduction. In particular 
under strong convexity and smoothness assumptions (to be specified later) we show that there exist positive constants $K_1$ and $K_2$ such that,
\begin{align}
\int_0^Te^{K_1\kappa(t-T)}\mathbb{E}\left[f(x_t^i)-f(x^*)\right]dt\leq K_2\left(\frac{1}{2}e^{-K_1\kappa T}+\frac{\sigma^2}{2N}+\frac{N-1}{N}\frac{d\sigma^2}{\underline{\lambda}+\kappa}\right),
\end{align}
where $\underline{\lambda}$ is the first nonzero eigenvalue of the graph Laplacian and $\kappa$ is the strong convexity constant related to the mirrored objective. A similar average regret bound can be established for the convex case (and $\kappa=0$).
This result, together with the technical assumptions, proofs and extensions will be presented in Propositions \ref{prop:conv_ismd_convex}-\ref{prop:conv_ismd_strconvex} and \ref{prop:fluctuation_stochastic}-\ref{prop:fluctuation_stochastic_strconvex}. The distance to optimality is thus bounded by three terms: 1) a term which decays exponentially with time, 2) a noise term which decreases as the number of particles increases 3) a term which arises from the distance between the particle values which is smaller for a connected graph, a stronger interaction and a larger strong convexity constant. 
\item We provide explicit rates of convergence and concentration inequalities for the particle system using log Sobolev inequalities.
This will be presented in Section \ref{sec:log_sob_N}. These results establish exponential rate of convergence in time for the law of the particle system, so that at equilibrium the samples $z_t^i$ will oscillate around the optimum $z^*$. The strong convexity constant plays a similar role as in the previous result, with a higher $\kappa$ resulting in a smaller distance to the optimum. 
\item Finally, we show the benefits of using interaction in convergence speed and variance reduction through different numerical experiments. In a distributed optimization setup we manage to achieve comparable performance and convergence to the full batch gradient descent, when using a "mini-batch" approach that uses $N$ particles with a mini-batch size that is $1/N$ times the total number of data-points or summands in $f$ (see Figure \ref{fig7} in Section \ref{sec:num} for more details).
\end{itemize}

The organization of this paper is as follows: Section \ref{sec:prelim} presents background material on MD and SMD together with some basic convergence results. Most of the material in this section is known, but we use this opportunity to set our framework and to relate the continuous time formulation with common discrete time implementations. We then propose ISMD in Section \ref{sec:ismd} and provide a detailed convergence analysis on how the particles approach the minimum as time increases. In Section \ref{sec:sample} we apply the Bakry-Emery theory for the law of the corresponding particle system. In Section \ref{sec:num} we demonstrate the performance of ISMD in practice and demonstrate the effect of interaction on decreasing the variance in a variety of examples. Finally, we provide some concluding remarks in Section \ref{sec:concl}.

\subsubsection{Notation} \label{notations}

We will use the following notations: we denote by $I_d$ the $d$-dimensional identity matrix and $1_d$ the $d$-dimensional vector of ones. Let $\textnormal{Diag}(a)$ with $a\in\mathbb{R}^d$ refer to a matrix with diagonal elements $[a_1,...,a_d]$. The Kronecker product is denoted by $\otimes$. Given an arbitrary norm $||\cdot||$ on $\mathbb{R}^d$, we will define $B_{||\cdot||}:=\{v\in\mathbb{R}^d:||v||\leq 1\}$. The dual norm $||\cdot||_*$ is defined as $||z||_*:=\sup\{z^Tv:v\in B_{||\cdot||}\}$. Denote with $||A||_2$ the spectral norm if $A$ is a matrix. Similarly, let $||A||_F$ denote the Frobenius norm given by $||A||_F^2=\sum_{i=1}^N\lambda_i^2$, where $\lambda_i$ denote the eigenvalues of $A$. Assume that the dual norm is compatible with the spectral norm, i.e. $||Az||_*\leq ||A||_2\; ||z||_*$. Unless specified otherwise $K,K'$ etc. denote generic constants whose value may change according to context.  

We will define a $\mu$ strongly convex function $f$ (w.r.t. a norm $\|\cdot\|$) if there exists $m>0$ such that for all $x,y$:
\begin{equation}
f(y)\geq f(x)+\nabla f(x)^{T}(y-x)+{\frac {m}{2}}\|y-x\|^{2}.
\end{equation}
The case $m=0$ corresponds to the convex case. Note that equivalently strong convexity can be defined using
\begin{align}
(x-y)^T(\nabla f(x)-\nabla f(y))\geq m \|x-y\|^2. 
\end{align} 
Furthermore, for a Lipschitz function  we let 
\begin{align}
\left\Vert f\right\Vert _{Lip}=\sup_{x\neq y}\frac{\left|f(x)-f(y)\right|}{\left\|x-y\right\|}<\infty.	
\end{align}
When looking at the Hessian of a convex function $f$ we will denote with $\textnormal{Hess}(f)\succeq \lambda I_d$ that $\textnormal{Hess}(f)- \lambda I_d$ is positive semidefinite with $\lambda\geq 0$.

For random variables $X,Y$ we say $X\overset{d}{=}Y$ to denote equality in distribution. For a measure $\pi$ and measurable function $f$ we use $\pi(f)=\mathbb{E}_{\mu}[f]=\int f(x)\pi(dx)$.
For any measurable space $\left(\mathcal{Z},\mathscr{B}(\mathcal{Z})\right)$
we use $\mathcal{P}(\mathcal{Z})$ to denote the space of all probability
measures on $\mathcal{Z}$ and $\mathcal{P}_2(\mathcal{Z})$ the one for finite second moments. A probability measure $\nu$ satisfies a Log-Sobolev inequality with constant $C$ is for any smooth function $f$ we have,
\begin{align}\label{eq:log_sobolev_def}
\textnormal{Ent}_\nu(f^2)\leq C\nu(|\nabla f|^2),
\end{align}
where the entropy is defined as,
\begin{align}
\textnormal{Ent}_\nu(f^2) = \nu(f^2\log f^2) - \nu(f^2)\log(\nu(f^2)).
\end{align}
For $\mu,\nu\in\mathcal{P}(Z)$ we will
denote the Kullback-Leibler divergence or relative entropy as,
\begin{align}
H(\mu|\nu)=\begin{cases}
\int\log\frac{d\mu}{d\nu}d\mu & \mbox{ if \ensuremath{\mu\ll\nu}},\\
+\infty & \mbox{ otherwise}.
\end{cases}
\end{align}
The 2-Wasserstein distance is defined as 
\begin{align}
W_{2}(\mu,\nu)=\left(\inf_{\Pi(\mu,\nu)}\left(\frac{1}{2}\int\left|z-z'\right|^{2}\Pi(dz,dz')\right)\right)^{1/2},
\end{align}
where $\Pi\in\mathcal{P}_2(\mathcal{Z}\times\mathcal{Z})$ is the coupling
of $\mu,\nu$ with $\Pi\left(\mathcal{Z},\cdot\right)=\nu$ and $\Pi\left(\cdot,\mathcal{Z}\right)=\mu$. 
We will assume that the optimization algorithm is started from a fixed deterministic point, and let $\mathbb{E}[\cdot]$ refer to the expected value conditional on the initial value. 

\section{Background}\label{sec:prelim}

\subsection{Preliminaries} 

We are interested in computing $x^*\in\arg\min_{x\in\mathcal{X}}\{f(x)\}$ under the assumption of smoothness and convexity for $f$ and $\mathcal{X}$. Throughout the paper we will assume continuity and smoothness of $f$:
\begin{assumption}\label{ass:f}
We let $f:\mathbb{R}^d\rightarrow\mathbb{R}$ be $L$-Lipschitz continuous with $L$-Lipschitz continuous gradients.  $\mathcal{X}\subset\mathbb{R}^d$ is a closed convex set.
\end{assumption}
We allow $f$ and $\nabla f$ to be Lipschitz w.r.t arbitrary norms. As all norms are equivalent up to a proportionality constant, a natural choice without loss of generality is the Euclidean norm, which is also self dual, i.e. $||\cdot||_*=||\cdot||_2$ (see \cite[Lemma 2.1.2]{nesterov2018lectures}). Despite this in the remainder we will distinguish between the chosen norm and its dual mainly in order to emphasize the difference when operating in the primal or mirror domain.

We will use a \emph{mirror map} $\Phi:\mathcal{X}\rightarrow\mathbb{R}^d$ to convert the constrained optimization problem to an unconstrained one and adopt the following standard assumptions (e.g. \cite[Assumption 9.3]{beck17book}):
\begin{assumption}\label{ass:mirror1}
$\Phi:\mathcal{X}\rightarrow\mathbb{R}^d$ is $\mu$ strongly convex and continuously differentiable. 
\end{assumption}
The mirror map and its conjugate will be used to pass between the constrained and unconstrained space. Define 
\begin{align}
\Phi^*(z) := \max_{x\in\mathcal{X}}(x^Tz-\Phi(x))
\end{align}
 to be the Legendre-Fenchel convex conjugate, so that
\begin{align}
\nabla \Phi^*(z):=\arg\max_{x\in\mathcal{X}}(z^Tx-\Phi(x)).
\end{align}
Note that one has
\begin{align}
\nabla\Phi\circ\nabla\Phi^*(z)=z.
\end{align} 
Note when $\Phi$ is strongly convex with constant $\mu$, the conjugate correspondence theorem (see Section 5.3.1 in \cite{beck17book})  gives that the gradient of $\Phi^*$ is Lipschitz continuous:
\begin{align}\label{eq:lip_phi_star}
||\nabla\Phi^*(z)-\nabla\Phi^*(z')||\leq \mu^{-1}||z-z'||_*.
\end{align}

We furthermore make the additional assumption that the conjugate of the mirror map maps $\mathbb{R}^d$ to $\mathcal{X}$:
\begin{assumption}\label{ass:mirror2}
$\nabla\Phi^*(\mathbb{R}^d)=\mathcal{X}$ and $||\Delta\Phi^*||_\infty<\infty$.
\end{assumption}
This simplifies the analysis as $\nabla\Phi^*$ maps directly to $\mathcal{X}$ and this avoids the need for projections. Extending our results without using Assumption \ref{ass:mirror2} is possible by following a route similar to \cite{mertikopoulos18}.

We will often use the notation $z^*$ for $z^*\in\mathbb{R}^n$ such that $x^*=\nabla\Phi^*(z^*)\in\arg\min_{x\in\mathcal{X}}f(x)$.

The \emph{Bregman divergence} is defined as,
\begin{align}
D_\Phi(x,y) = \Phi(x)-\Phi(y)-\nabla\Phi(y)^T(x-y).
\end{align}
The Bregman divergence is meant to quantify how far a point $x$ is from $y$ and $\Phi$
can be thought as a distance generating function that adapts to the geometry or structure of $\mathcal{X}$. Two well-known examples of the Bregman divergence are the Euclidean distance with $\Phi(x)=\frac{1}{2}||x||_2^2$, and the simplex constraint $\Phi(x) = \sum_{i=1}^dx(i)\log x(i)-x(i)$. 

An important property of using a mirror map with the Bregman divergence is that 
\[
D_{\Phi^*}(z_t,z^*)=D_\Phi(x^*,x_t),
\]
for some values $z_t$ and $z^*$ ($x_t$ and $x^*$) in unconstrained (constrained) space. 
In addition, the $\Phi$-diameter of $\mathcal{X}$ is defined as
\begin{align}\label{eq:diameter}
D_{\Phi,\mathcal{X}}:=\sup_{x\in\mathcal{X},x'\in\mathcal{X}^o}\sqrt{2 D_\Phi(x,x')},
\end{align} 
and is meant to quantify the size of the constrained space $\mathcal{X}$ when $D_\Phi(x,y)$  is used as a measure of distance. 

We will make the following assumption on the Lipschitz-continuity on the Bregman divergence. 
\begin{assumption}\label{ass:lipschitz_bregman}
The Bregman divergence satisfies a Lipschitz condition of the form,
\begin{align}
|D_{\Phi^*}(x,z)-D_{\Phi^*}(y,z)|\leq K ||x-y||_*.
\end{align}
for some constant $K$. 
\end{assumption}
This is a common technical condition (e.g. see \cite{shahrampour2017distributed}), which is automatically satisfied when the function $\Phi^*$ is Lipschitz. Here it is used only for the strongly convex case in Proposition \ref{prop:conv_ismd_strconvex}. 

We will assume convexity of the objective function $f$.
\begin{assumption}\label{ass:f_convex}
$f:\mathbb{R}^d\rightarrow\mathbb{R}$ is convex.
\end{assumption}
This further implies that the gradient of $f$ is uniformly bounded on $\mathcal{X}$ by the constant $L$, i.e. $||\nabla f(x)||_*\leq L$. 
Alternatively, at times we will work under a strong convexity assumption. 
\begin{assumption}\label{ass:f_strconvex}
The function $f:\mathbb{R}^d\rightarrow\mathbb{R}$ is a $\mu$-strongly convex function with respect to $\Phi$. This means \[D_f(x,y)\geq \mu D_{\Phi}(x,y).\]
\end{assumption} 

We further let the anti-derivative $\mathcal{V}$ be defined as
$\nabla\mathcal{V}=\nabla f\circ \nabla\Phi^*$. We note that knowledge of $\mathcal{V}$ here is mainly used to improve exposition, convey intuition and in practice we work only with $\nabla\mathcal{V}$. The main use of $\mathcal{V}$ itself will be to describe invariant distributions of the mirror descent SDEs as defined in Section \ref{sec:mddef}. For the analysis, we note that Assumptions \ref{ass:f} and \ref{ass:mirror1} imply $\nabla\mathcal{V}$ is Lipschitz and we will use the following convexity assumptions in Section \ref{sec:bound_fluct} to bound the fluctuations of the particle system and in Section \ref{sec:sample} to establish rates of convergence. 
We will distinguish between the convex and strongly convex cases for $\mathcal{V}$.
\begin{assumption}
\label{ass:V_convex}
The function $\mathcal{V}$ is convex.
\end{assumption} 
\begin{assumption}
\label{ass:V_strconvex}
The function $\mathcal{V}$ is $\kappa$- strongly convex
with $\kappa>0$. 
\end{assumption} 

\subsection{Mirror Descent}\label{sec:mddef}

We begin the presentation of mirror descent (MD) with some preliminaries on gradient descent. We provide a brief account of the main approaches in discrete time and later discuss the continuous time setting, which is of interest in this paper.

\subsubsection{Projected Gradient Descent}
One way of finding the solution of a constrained optimization problem is through projected gradient descent (GD), 
\begin{align}\label{proj_gd}
x_{k+1} = \Pi_\mathcal{X}(x_k-\eta\nabla f(x_k)),
\end{align}
where $k$ here denotes discrete time, $\eta$ is the learning rate and we define the projection using the Euclidean norm, i.e. $\Pi_\mathcal{X}(y) = \arg\min_{x\in\mathcal{X}}||y-x||_2^2$. The main drawback of this method is its slow convergence: when $d$, the dimension of $x$, increases the convergence can slow down. This is due to the fact that this scheme is tied to the Euclidean geometry of $\mathbb{R}^d$ through the projection operator $\Pi_\mathcal{X}$. Consider a setting in which $||\nabla f(x)||_\infty\leq 1$. This implies $||\nabla f(x)||_2<\sqrt{d}$, so projected gradient descent will converge at a rate $\sqrt{d/k}$ (see e.g. Example 9.17 from \cite{beck17book}).

\subsubsection{Mirror Descent in Discrete and Continuous Time}\label{MD}
A generalization of the projected gradient descent method is mirror descent, which is given by:
\begin{align}
&\nabla \Phi(y_{k+1}) = \nabla\Phi(x_k) - \eta\nabla f(x_t),\\
& x_{k+1} = \Pi_\mathcal{X}^\Phi(y_{k+1}) := \arg\min_{x\in\mathcal{X}} D_\Phi(x,y_{k+1}).
\end{align}
The last step can be rewritten as,
\begin{align}
 x_{k+1} =&  \arg\min_{x\in\mathcal{X}}\{\Phi(x) - \nabla\Phi(y_{k+1})^Tx\}.
\end{align}

The current point $x_k$ is thus mapped into the dual space $\nabla\Phi(x_k)$; this is then updated to $\nabla \Phi(y_{k+1})$ by stepping in the direction of the negative gradient $-\eta\nabla f(x_k)$ and then mapped back to $x_{k+1}$.  Observe that for $\Phi(x) = \frac{1}{2}||x||_2^2$, we have $\nabla\Phi(x) = x$ and the Bregman divergence is simply given by $D_\Phi(x,y) = \frac{1}{2}||x-y||_2^2$, so that the algorithm is equivalent to projected gradient descent. Consider $
\Phi(x) = \sum_{i=1}^nx(i)\log x(i)$ and let $\mathcal{X}:=\Delta_d$, the $d$-dimensional simplex. Observe that $\nabla_i \Phi(x) = \log x(i) + 1$, and the projection of $y$ with respect to this Bregman divergence onto the simplex $\Delta_d$ amounts to a scaling with the $L^1$-norm. The benefit of MD is that it can adapt to the structure or geometry of $\mathcal{X}$ and hence result in a much better dependence of the rate of convergence on the problem dimension $d$ (see Chapter 9 in \cite{beck17book}).



Using Assumption \ref{ass:mirror2} the continuous time version of mirror descent is given by,
\begin{align}\label{eq:ctmd}
dz_t = -\nabla f(x_t)dt, \;\;\; x_t = \nabla \Phi^*(z_t),
\end{align}
where we have set $\eta=1$ for simplicity.

\subsubsection{Convergence of MD in continuous time}

We will proceed with presenting two well-known convergence results for deterministic mirror descent and the dynamics of \eqref{eq:ctmd}; see \cite{bubeck2014convex}, \cite{nemirovsky83}) for more details. These results are based on using Bregman divergence in a Lyapunov function $V_t$ and are included to ease exposition and to facilitate the comparison to the stochastic case. Note that $D_{\Phi,\mathcal{X}}$ is defined in \eqref{eq:diameter}. 

\begin{lemma}[Convergence of MD for a convex objective]\label{lem:conv_md_convex}
Assume that Assumptions \ref{ass:f} -- \ref{ass:mirror2} and \ref{ass:f_convex} hold. We then have the following result,
\begin{align}
\frac{1}{T}\int_0^T(f(x_t)-f(x^*))dt\leq \frac{D_{\Phi,\mathcal{X}}^2}{2T}.
\end{align}
\end{lemma}
\begin{proof}
 Consider, 
\begin{align}\label{eq:deflyap}
V(z_t):=\Phi^*(z_t)-\Phi^*(z^*)-\nabla\Phi^*(z^*)^T(z_t-z^*).
\end{align}
Observe that $V(z^*)=0$ and $V(z_t)\geq 0$ for $z^*\neq z_t$. Furthermore, 
\begin{align}\label{eq:dvconvex}
\frac{dV(z_t)}{dt} &= (\nabla\Phi^*(z_t)^T-\nabla\Phi^*(z^*)^T)\frac{dz_t}{dt}\\
&=- (x_t-x^*)^T\nabla  f(x_t)\leq f(x^*)-f(x_t)\leq 0,
\end{align} 
where in the second equality we used $x_t=\nabla\Phi^*(z_t)$ and in last inequality we have used the convexity of $f$. Integrating $dV(z_t)/dt$ from $t=0$ to $t=T$, using the convexity of $f$ and the relationship $D_{\Phi^*}(z_t,z^*)=D_\Phi(x^*,x_t)$ so that $V(z_0)\leq \frac{1}{2}D_{\Phi,\mathcal{X}}^2$ we obtain,
\begin{align}
\int_0^T (f(x_t)-f(x^*))dt \leq \frac{1}{2}D_{\Phi,\mathcal{X}}^2,
\end{align}
and the claim follows. 
\end{proof}
Observe that furthermore since $\inf_{0\leq t\leq T}f(x_t)\leq \frac{1}{T}\int_0^T f(x_t)dt$, we can obtain a result for the best value of $x_t$,
\begin{align}\label{eq:convratemd}
\inf_{0\leq t\leq T}f(x_t)-f(x^*)\leq \frac{D_{\Phi,\mathcal{X}}^2}{2T}.
\end{align}

Observe that for projected gradient descent we have $\Phi(x) = \frac{1}{2}||x||_2^2$. In this case, $D_{||\cdot||_2^2,\mathcal{X}}=\sup_{x,x'\in\mathcal{X}}||x-x'||_2$.
From Lemma \ref{lem:conv_md_convex} we observe that if $D^2_{\Phi,\mathcal{X}}$ is smaller than $D_{||\cdot||_2^2,\mathcal{X}}$ then mirror descent algorithm could achieve a faster convergence rate than gradient descent due to the smaller constant. 

In the strongly convex case the convergence speed can be increased to an exponential rate. 

\begin{lemma}[Convergence of MD for a strongly convex objective]\label{lem:conv_md_strconvex}
Assume that Assumptions \ref{ass:f} -- \ref{ass:mirror2} and \ref{ass:f_strconvex} hold. We then have the following result,
\begin{align}
D_\Phi(x^*,x_t)\leq e^{-\mu t}D_\Phi(x^*,x_0).
\end{align}
\end{lemma}
\begin{proof}
Since $f$ is $\mu$-strongly convex with respect to $\Phi$, it holds,
\begin{align}\label{eq:mustrcon}
D_f(x,y)\geq  \mu D_{\Phi}(x,y).
\end{align}
Observe that,
\begin{align}
\frac{d D_{\Phi^*}(z_t,z^*)}{dt}&=-(x_t-x^*)^T\nabla f(x_t)\\
&= - D_f(x^*,x_t) -f(x_t)+f(x^*)\\
&\leq -\mu D_{\Phi^*}(z_t,z^*)-f(x_t)+f(x^*)\\
&\leq -\mu D_{\Phi^*}(z_t,z^*),\label{eq:dphistr}
\end{align}
where in the second equality we have used the definition of the Bregman divergence, in the second-to-last inequality we have used \eqref{eq:mustrcon} and the fact that $D_\Phi(x^*,x_t)=D_{\Phi^*}(z_t,z^*)$, and in the last inequality we have used the convexity of $f$, namely that $f(x^*)-f(x_t)<0$. 
Let now,
\begin{align}
V(t,z_t):=e^{\mu t} D_{\Phi^*}(z_t,z^*).
\end{align}
Then, 
\begin{align}
\frac{d V(t,z_t)}{dt}&=\mu e^{\mu t}D_{\Phi^*}(z_t,z^*)+e^{\mu t} \frac{d D_{\Phi^*}(z_t,z^*)}{dt}\\
&\leq \mu e^{\mu t}D_{\Phi^*}(z_t,z^*)-\mu e^{\mu t} D_{\Phi^*}(z_t,z^*)\leq 0.
\end{align}
where in the second-to-last inequality we have used \eqref{eq:dphistr}. 
Then, since $\frac{dV(t,z_t)}{dt}\leq 0$ it holds that $V(t,z_t)\leq V(0,z_0)$. Using this and $D_\Phi(x^*,x_t)=D_{\Phi^*}(z_t,z^*)$ we have,
\begin{align}
V(t,z_t)&=e^{\mu t}D_{\Phi^*}(z_t,z^*)=e^{\mu t}D_{\Phi}(x^*,x_t)\\
&\leq V(0,z_0)=D_{\Phi}(x^*,x_0).
\end{align}
This implies,
\begin{align}
D_{\Phi}(x^*,x_t) \leq e^{-\mu t} D_{\Phi}(x^*,x_0).
\end{align}
\end{proof}

\subsection{Stochastic mirror descent}

Consider adding noise to \eqref{proj_gd}: 
\begin{align}
x_{k+1} = \Pi_\mathcal{X}\left(x_k-\eta\epsilon\nabla f(x_k)+\sqrt{\epsilon}\sigma \zeta_k\right),\quad \zeta_k\sim \mathcal{N}(0,I_d),\:\text{i.i.d.}
\end{align} 
In the limit of the time step $\epsilon\downarrow 0$ one recovers the following  It\^o stochastic integral equation \cite{PETTERSSON2000125}:
\begin{align}
& x_{t+1}= x_0-\eta\int_0^t\nabla f(x_s)ds+\sigma B_t +u_t, \label{eq:proj_sde} \\
& u_t=\int_0^t 1_{x_t\in\partial\mathcal{X}} du_s.
\end{align}
Here  $B_t$ is a $d$-dimensional Brownian motion and a local time construction on the boundary of $\mathcal{X}$ is used for $u_t$ to ensure that $X_t\in\mathcal{X}$ almost surely; see \cite{PETTERSSON2000125,Storm95} and the references therein for details. Understanding in detail the ergodicity properties or long time behaviour of \eqref{eq:proj_sde} is largely unexplored in the literature. We note that at the level of the evolution of densities, reflecting boundary conditions are commonly used for the forward or Fokker-Planck equation, e.g. \cite[Section 4.2.2]{pavliotis14book}, but a detailed study of \eqref{eq:proj_sde} along these lines goes beyond the scope of this article. In the remainder we will consider the mapping to be on the mirror domain.

The stochastic mirror descent (SMD) is given by the SDE,
\begin{align}\label{eq:smd}
dz_t = -\nabla f\circ\nabla\Phi^*(z_t)dt + \sigma dB_t,
\end{align}
where $x_t=\nabla\Phi^*(z_t)$ and as before in \eqref{eq:ctmd} we use $\eta=1$. Note that when $\sigma=0$ we obtain a deterministic variant of continuous time mirror descent. The long time behavior of such dynamics are very well understood. The law of $z_t$ will converge to an invariant distribution with density proportional to $\exp(-\frac{2}{\sigma^2}\mathcal{V})$ where recall $\mathcal{V}$ denotes the anti-derivative of $\nabla f\circ\nabla\Phi^*$. We will discuss this in more detail later in Section \ref{sec:sample}. The downside of this dynamics is that SMD cannot converge to the exact solution, i.e. the noise keeps the algorithm from fully converging exactly to $x^*$ (see \cite{raginsky12}, \cite{mertikopoulos18}). The latter would require the equilibrium distribution being a Dirac/atomic measure, which in turn requires using a decreasing noise in time; see  \cite{mertikopoulos18} for more details.
 
Another interesting remark is that the dynamics in the dual space is then equivalent to that of SGD with objective function $\mathcal{V}$. It is well-known that for a convex objective SGD methods converge in $O(1/t)$ to a neighbourhood of the optimum the size of which is proportional to the noise variance \cite{gower2019sgd}. A similar result can be shown to hold for SMD (see e.g. \cite{raginsky12}, \cite{wilson2018lyapunov}, \cite{mertikopoulos18}). 
\begin{remark}[SGD and additive noise]
We note that in this work we refer to SGD as the SDE equivalent of gradient descent with additive noise, i.e. we use a constant diffusion coefficient. This is in disagreement with recent works in the areas of machine learning whereby the noise in SGD arises only from sub-sampling a deterministic loss function. The latter can be modelled asymptotically as an SDE with multiplicative noise and state-dependent diffusion; see for instance \cite{li2017stochastic,li2019stochastic,mandt2017stochastic}.
\end{remark} 

\subsubsection{Convergence of SMD}

We proceed with presenting the analogs of Lemmata \ref{lem:conv_md_convex}-\ref{lem:conv_md_strconvex} for the SMD case.

\begin{proposition}[Convergence of SMD for a convex objective]\label{prop:conv_smd_convex}
Assume that Assumptions \ref{ass:f} -- \ref{ass:mirror2} and \ref{ass:f_convex} hold. We then have the following result, 
\begin{align}
\mathbb{E}\left[\frac{1}{T}\int_0^T(f(x_t) -f(x^*))dt\right] \leq& \frac{1}{2T}D^2_{\Phi,\mathcal{X}} + \frac{1}{2}\sigma^2  ||\Delta\Phi^*||_\infty.
\end{align}
\end{proposition}
\begin{proof}
We define the following Lyapunov function, 
\begin{align}\label{eq:deflyap2}
V(z_t):=\Phi^*(z_t)-\Phi^*(z^*)-\nabla\Phi^*(z^*)^T(z_t-z^*).
\end{align}
We then have the following results. 
By It\^o's lemma we obtain, 
\begin{align}
dV(z_t)=& \left(\nabla f(x_t)^T(x^*-x_t)+\frac{1}{2}\sigma^2\Delta\Phi^*(z_t)\right)dt+ \sigma (x_t-x^*)^TdB_t,
\end{align}
where for the last term we have $\sigma (x_t-x^*)^TdB_t\overset{d}{=}\sigma ||x_t-x^*||_2dB_t$. Integrating the expression we obtain,
\begin{align}
V(z_T)=V(z_0)+\int_0^T (x^*-x_t)^T\nabla f(x_t) dt + \frac{1}{2}\sigma^2\int_0^T\Delta\Phi^*(z_t)dt + \sigma\int_0^T||x_t-x^*||_2dB_t.
\end{align}
Observe again that, $(x^*-x_t)^T\nabla f(x_t)\leq f(x^*)-f(x_t)$,
by convexity of $f$. Furthermore, as before $V(z_0)\leq \frac{1}{2}D^2_{\Phi,\mathcal{X}}$. Lastly, we have $\frac{1}{2}\sigma^2\int_0^T\Delta\Phi^*(z_t)dt\leq \frac{1}{2}\sigma^2 T ||\Delta\Phi^*||_\infty$. Rearranging we then obtain,
\begin{align}
\int_0^T (f(x_t) -f(x^*))dt \leq \frac{1}{2}D^2_{\Phi,\mathcal{X}} + \frac{1}{2}\sigma^2 T ||\Delta\Phi^*||_\infty + \sigma\int_0^T||x_t-x^*||_2dB_t.
\end{align}
Using the fact that the process $\{||x_t-x^*||_2\}_{t\geq 0}$ is adapted to the filtration of $x_t$ so that \begin{align}
\mathbb{E}\left[\sigma\int_0^T||x_t-x^*||_2dB_t\right]=0,
\end{align} 
we have,
\begin{align}
\mathbb{E}\left[\frac{1}{T}\int_0^T (f(x_t) -f(x^*))dt\right] \leq \frac{1}{2T}D^2_{\Phi,\mathcal{X}} + \frac{1}{2}\sigma^2  ||\Delta\Phi^*||_\infty
\end{align}
\end{proof}

A similar bound can be obtained for the time average $\mathbb{E}\left[f\left(\frac{1}{T}\int_0^T x_tdt\right) -f(x^*)\right] $ by using Jensen's inequality.
Similar to the deterministic case, if $D^2_{\Phi,\mathcal{X}}$ is smaller than $D_{||\cdot||_2^2,\mathcal{X}}$ then MD achieves a faster convergence rate than gradient descent.
 
\begin{proposition}[Convergence of SMD for a strongly convex objective]\label{prop:conv_smd_strconvex}
Assume that Assumptions \ref{ass:f} -- \ref{ass:mirror2} and \ref{ass:f_strconvex} hold.
Then,
\begin{align}
\mathbb{E}\left[D_{\Phi}(x^*,x_T)\right]\leq &e^{-\mu T} \frac{1}{2}D_{\Phi,\mathcal{X}}^2 + \frac{1}{2\mu}\sigma^2 (1-e^{-\mu T})||\Delta \Phi^*||_{\infty}.
\end{align}
\end{proposition}
\begin{proof}
Let $V(t,z_t)=e^{\mu t}D_{\Phi^*}(z_t,z^*)$. We have through It\^o's lemma, 
\begin{align}
dV(t,z_t) = \left(e^{\mu t}(x^*-x_t)^T\nabla f(x_t)+\mu e^{\mu t} D_{\Phi^*}(z_t,z^*)+ \frac{1}{2}\sigma^2e^{\mu t}\Delta \Phi^*(z_t)\right)dt + \sigma e^{\mu t} (x_t-x^*)^Td B_t. 
\end{align}
Integrating the expression we obtain,
\begin{align}
V(T,z_T)=&V(0,z_0)+ \int_0^T e^{\mu t} \left((x^*-x_t)^T\nabla f(x_t)+ \mu D_{\Phi^*}(z_t,z^*)\right)dt + \int_0^T \frac{1}{2}\sigma^2e^{\mu t}\Delta \Phi^*(z_t) dt \\
&+ \int_0^T \sigma e^{\mu t} (x_t-x^*)^Td B_t.
\end{align}
Observe again that for a $\mu$-strongly convex function $D_f(x^*,x_t)\geq \mu D_\Phi(x^*,x_t)$ and by the properties of the mirror map $D_\Phi(x^*,x_t)=D_{\Phi^*}(z_t,z^*)$. This implies that $\nabla f(x_t)^T(x^*-x_t) \leq -D_f(x^*,x_t)-f(x_t)+f(x^*)\leq -\mu D_\Phi(x^*,x_t)=-\mu D_{\Phi^*}(z_t,z^*)$. Furthermore, $\int_0^T \sigma e^{\mu t} (x_t-x^*)^Td B_t\overset{d}{=}\int_0^T \sigma e^{\mu t} ||x_t-x^*||_2d B_t$ and it holds that $\int_0^T \frac{1}{2}\sigma^2 e^{\mu t} \Delta \Phi^*(z_t) dt\leq \frac{1}{2}(e^{\mu T}-1)\sigma^2||\Delta \Phi^*||_{\infty}$. Using these,
\begin{align}
V(T,z_T) &\leq V(0,z_0) + \frac{1}{2}\sigma^2 (e^{\mu T}-1) ||\Delta \Phi^*||_{\infty} + \int_0^T \sigma e^{\mu t} ||x_t-x^*||_2d B_t.
\end{align}
Taking expected values we obtain,
\begin{align}
\mathbb{E}\left[D_{\Phi}(x^*,x_t)\right]\leq e^{-\mu T} D_{\Phi,\mathcal{X}}^2 + \frac{1}{2}\sigma^2 (1-e^{-\mu T})||\Delta \Phi^*||_{\infty},
\end{align}
where we have used that the expected value of an It\^o integral is zero and the relationship $D_{\Phi^*}(z_t,z^*)=D_{\Phi}(x^*,x_t)$.
\end{proof}

Again in the strongly convex case the convergence speed can be increased. However, as expected for both the convex and strongly convex results, as $T\rightarrow\infty$, the gap to optimality is bounded from above by a quantity proportional to the noise variance $\sigma^2$. The solution will not converge exactly to the minimum, but oscillate around it.

\section{Interacting Stochastic Mirror Descent}\label{sec:ismd}

It is clear that when using a fixed learning rate and a constant noise variance, SGD does not converge to the optimum. In fact, the distance to the optimum is controlled by the amount of noise. Various strategies such as using a vanishing noise variance by increasing the batch size or variance-reduced SGD (see e.g. \cite{johnson2013accelerating}, \cite{defazio2014saga}, \cite{gorbunov2019unified}) have been proposed. In this work, we consider an alternative approach for controlling the distance from the optimum; namely by using interactions between the particles \cite{raginsky12}. 

We consider as an alternative to \eqref{eq:smd} and we consider the following interacting particle dynamics,
\begin{align}\label{eq:cmdpart}
dz_t^i = -\nabla f\circ\nabla \Phi^*(z_t^i) dt + \sum_{j=1}^N A_{ij}(z_t^j-z_t^i) dt + \sigma dB_t^i,\quad i=1,...,N,
\end{align}
where $B_t^i$ are independent Brownian motions and $A=\{A_{ij}\}_{i,j=1}^N$ is an $N\times N$ doubly-stochastic matrix representing the interaction weights. As mentioned in Section \ref{sec:intro} the matrix $A$ represents an interaction graph which can also impose communication constraints on the particles. Each particle $i$ will be influenced only from particles $j$ for which $A_{ij}\neq 0$. 

The discretized version of \eqref{eq:cmdpart} is then,
\begin{align}
&x_k^i = \nabla \Phi^*(z_k^i), \label{eq:imddisc_1}\\
&z_{k+1}^i = z_k^i - \epsilon\nabla f(x_k^i) + \epsilon \sum_{j=1}^N A_{ij} (z_k^j-z_k^i) + \sqrt{\epsilon}\sigma\zeta_k^i, \label{eq:imddisc}
\end{align}
where $\zeta_t\sim\mathcal{N}(0,I_d)$ and $\epsilon$ is the discretization parameter. In the absense of interactions, i.e. when $A_{ij}=0$, we obtain a discretized version of SMD. 

We remark here that the interacting mirror descent algorithm is equivalent to an algorithm in which the particle interaction is defined using a Bregman divergence instead of the $L^2$ distance. In particular for the time discretized version \eqref{eq:imddisc_1}-\eqref{eq:imddisc} we have the following result. 
\begin{lemma}[Interacting mirror descent as Bregman distance interaction in primal space]\label{lem:imd_bregman}
Let $\sigma=0$ and $\epsilon=1$, then interacting mirror descent in \eqref{eq:imddisc_1}-\eqref{eq:imddisc} can be rewritten as:
\begin{align}
x_{k+1}^i = \arg\min_{x\in\mathcal{X}} \{\nabla f(x_k^i)^Tx  +\sum_{j=1}^N A_{ij}D_\Phi(x,x_{k}^j)\}.
\end{align}
\end{lemma}
\begin{proof}
 We have,
\begin{align}
x_{k+1}^i &= \nabla \Phi^*(z_{k+1}^i)\\
&=\arg\min_{x\in\mathcal{X}}\{D_\Phi(x,\nabla \Phi^*(z^i_{k+1}))\}\\
&=\arg\min_{x\in\mathcal{X}} \{\Phi(x)-\Phi(x_{k+1}^i)-\nabla\Phi\circ\nabla\Phi^*(z^i_{k+1})^T(x-\nabla\Phi^*(z^i_{k+1}))\}\\
&=\arg\min_{x\in\mathcal{X}} \{\Phi(x)-\Phi(x_{k+1}^i)-(z^i_{k+1})^Tx\},
\end{align}  
where in the last equality we have used $\nabla\Phi(\nabla\Phi^*(x))=x$.
Continuing,
\begin{align}
x_{k+1}^i&=\arg\min_{x\in\mathcal{X}}\{ \Phi(x)-\Phi(x_{k+1}^i)+\nabla f(x_k^i)^Tx  -\sum_{j=1}^N A_{ij} (z_k^j)^Tx\}\\
&=\arg\min_{x\in\mathcal{X}} \{\nabla f(x_k^i)^Tx  +\sum_{j=1}^N A_{ij} \left(\Phi(x)-\Phi(x_{k+1}^i)-(z_k^j)^Tx\right)\},
\end{align}
where in the first equality we have used the evolution of $z_{t+1}^i$ (note: the $z_t^i$ terms cancel out due to the double stochasticity of $A_{ij}$).
Then,
\begin{align}
x_{k+1}^i&=\arg\min_{x\in\mathcal{X}}\{ \nabla f(x_k^i)^Tx  +\sum_{j=1}^N A_{ij}\left(\Phi(x)-\Phi(x_{k}^j)- (\nabla\Phi(x^j_{k}))^T(x-x_k^i)\right)\}\\
&=\arg\min_{x\in\mathcal{X}} \{\nabla f(x_k^i)^Tx  +\sum_{j=1}^N A_{ij}D_\Phi(x,x_{k}^j)\},
\end{align}
where in the first equality we have used $z_k^j = \nabla\Phi(x_k^j)$, and -- due to the minimum being taken over $x$ -- replaced $\Phi(x_{k+1}^i)$ with $\Phi(x_k^j)$ and added the term $\nabla\Phi(x_{k}^j)^Tx_k^j$ and in the last equality we have used the definition of the Bregman divergence. 
\end{proof}

Unlike the standard $L_2$ consensus algorithm where $x_{t+1}^i =\arg\min_{x\in\mathcal{X}} \nabla f(x_t^i)^Tx  +\sum_{j=1}^N A_{ij} ||x-x_t^j||^2_2$, here the particles interact in the Bregman distance. Lemma \ref{lem:imd_bregman} can be combined with the analysis in \cite{PETTERSSON2000125} to extend \eqref{eq:proj_sde} to an interacting reflected SDE like \eqref{eq:cmdpart} that uses $\sum_{j=1}^N A_{ij}D_\Phi(x_t^i,x_{t}^j)dt$ as an interaction term. This is useful for conveying intuition, but the analysis in the mirror domain is much simpler. One does not have to deal with terms like $u_t$ in \eqref{eq:proj_sde}, and the interaction term in \eqref{eq:cmdpart} is symmetric, which simplifies the analysis, see \cite{malrieu01} for a more detailed discussion.

\subsection{A useful reparameterization}\label{sec:vec_form_ismd}

It will often be useful to represent the dynamics of the particle system presented in \eqref{eq:cmdpart} as a vector SDE with all particles stacked in a single vector variable $\mathbf{z}_t = ((z_t^1)^T,...,(z_t^N)^T)^T$. We define the graph Laplacian as $L:=\textnormal{Diag}(A\mathbf{1}_N)-A$, and let $\mathcal{L}:=L\otimes I_d$, where $\otimes$ is the Kronecker product. 
Using the Laplacian, we can rewrite the evolution of the $z_t^i$-s in vector form as,
\begin{align}\label{eq:cmdpart_vec}
d\mathbf{z}_t =\left( -\nabla\mathbf{V}(\mathbf{z}_t)-\mathcal{L}\mathbf{z}_t\right)dt + \sigma d\mathbf{B}_t,
\end{align}
where $\mathbf{B}_t:=((B_t^1)^T,...,(B_t^N)^T)^T$ is the stacked variable of Brownian motions and $\nabla\mathbf{V}(\mathbf{z}_t)={(\nabla\mathcal{V}(z_t^1)^T,...,\nabla\mathcal{V}(z_t^N)^T)}^T$. Note that the interaction is linear and the distinctive properties of ISMD compared to $N$ independent copies of SMD are contained $\mathcal{L}$. 
Note that given $A$ is doubly stochastic we have
\begin{align}\label{eq:L_eig_zero}
\sum_{i=1}^N\sum_{j=1}^N A_{ij}z^i = \sum_{i=1}^N z^i
\end{align}
so the smallest eigenvalue of $L$ is zero and  $\mathcal{L}1_{dN}=0$. We will return to this point in Section \ref{sec:sample}. 
We assume throughout that the network graph corresponding to the Laplacian is connected, which in turn implies that $\mathcal{L}$ has eigenvalues (\cite{mesbahi2010graph}):
\begin{align}
\lambda_0=0 < \underline{\lambda}\leq \lambda_3\leq ... \leq \lambda_{dN}.
\end{align}
We thus denote by $\underline{\lambda}$ the smallest non-zero eigenvalue of the Laplacian. 

\subsection{Convergence of ISMD: a general bound}\label{sec:gen_bound_ismd}
As we saw earlier, the SMD dynamics does not converge to the minimum due to the noise in the optimization algorithm. Here we show that introducing interaction between the particles can reduce the effect of noise. 

We will decompose each particle as a sum of the particle average and a fluctuation term, $z_t^i:=\bar z_t^N+\tilde z_t^i$, where we let
\begin{align}
\bar z_t^N := \frac{1}{N}\sum_{i=1}^N z_t^i,\quad \tilde z_t^i:=z_t^i-\bar z_t^N,
\end{align}
and also define $y_t^N:=\nabla\Phi^*(\bar z_t^N)$. 

We begin by deriving the evolution of the average of the particles $\bar z_t^N$. Observe,
\begin{align}
d\bar z_t^N &=-\frac{1}{N}\sum_{i=1}^N\nabla f(x_t^i)dt + \frac{1}{N}\sum_{i=1}^N\sum_{j=1}^NA_{ij}(z_t^j-z_t^i)dt + \frac{\sigma}{N}\sum_{i=1}^NdB_t^i\\
&=-\frac{1}{N}\sum_{i=1}^N\nabla f(x_t^i)dt + \frac{\sigma}{N}\sum_{i=1}^NdB_t^i,
\end{align}
where we have used the fact that $\frac{1}{N}\sum_{i=1}^N\sum_{j=1}^NA_{ij}z_t^j = \bar z_t^N$ due to the matrix $A$ being doubly stochastic. The particle average $\bar{z}_t$  moves along the gradient of $f$ towards the minimum $x^*$. At the same time the interaction aims to control fluctuations of each particle $z^i_t$ around  $\bar{z}_t$. This will appear clearly in the bounds in the subsequent results. More specifically, we obtain a result that the distance of a particle to the optimum is bounded by terms standard to optimization with the variance reduced by a factor of $N$ and terms related to the fluctuation, which we will show is bounded. 


We proceed by presenting the analog of Proposition \ref{prop:conv_smd_convex} for the ISMD.
\begin{proposition}[Convergence of ISMD for a convex objective]\label{prop:conv_ismd_convex}
Assume that Assumptions \ref{ass:f} -- \ref{ass:mirror2} and \ref{ass:f_convex} hold. We then have,
\begin{align}
\frac{1}{T}\int_0^T\mathbb{E}\left[(f(x_t^i)-f(x^*))\right]dt &\leq \frac{1}{2T}D^2_{\Phi,\mathcal{X}} + \frac{\sigma^2}{2N}||\Delta \Phi^*||_\infty+ \int_0^T \frac{L}{\mu T} \mathbb{E}\left[|| \tilde z_t^i||_*\right]dt\\
& + \int_0^T \frac{2L}{\mu N T}\sum_{i=1}^N  \mathbb{E}\left[|| \tilde z_t^i||_*\right]dt.
\end{align}
\end{proposition}
\begin{proof}
Observe that,
\begin{align}
\int_0^T (f(x_t^i)-f(x^*))dt &=\int_0^T(f(y_t^N)-f(x^*))dt + \int_0^T(f(x_t^i)-f(y_t^N))dt\\
&\leq \int_0^T(f(y_t^N)-f(x^*))dt + \int_0^TL||x_t^i-y_t^N||dt\\
&\leq \int_0^T (f(y_t^N)-f(x^*))dt + \int_0^T \frac{L}{\mu} || \tilde z_t^i||_*dt\label{eq:derdetN1},
\end{align}
where we have used the $L$-Lipschitz continuity of $f$ and the conjugate correspondence theorem for $\nabla \Phi^*$.
Furthermore,
\begin{align}
&\int_0^T (f(y_t^N)-f(x^*))dt = \int_0^T\frac{1}{N}\sum_{i=1}^N(f(x_t^i)-f(x^*))dt + \int_0^T\frac{1}{N}\sum_{i=1}^N(f(y_t^N)-f(x_t^i))dt\\
&\leq \int_0^T \frac{1}{N}\sum_{i=1}^N\left(f(x_t^i)-f(x^*)+\frac{L}{\mu} || \tilde z_t^i||_*\right)dt \label{eq:derdetN2},
\end{align}
where we have again used the Lipschitz continuity of both $f$ and $\nabla\Phi^*$. 
We furthermore have, by convexity of $f$,
\begin{align}
&\int_0^T \frac{1}{N}\sum_{i=1}^N(f(x_t^i)-f(x^*)dt\leq \int_0^T\frac{1}{N}\sum_{i=1}^N(x_t^i-x^*)^T\nabla f(x_t^i)dt \\
&\leq \int_0^T \frac{1}{N}\sum_{i=1}^N \nabla f(x_t^i)^T(x_t^i-y_t^N)dt + \int_0^T \frac{1}{N}\sum_{i=1}^N \nabla f(x_t^i)^T(y_t^N-x^*)dt\\
&\leq \int_0^T \frac{L}{\mu N}\sum_{i=1}^N||\tilde z_t^i||_*dt + \int_0^T \frac{1}{N}\sum_{i=1}^N \nabla f(x_t^i)^T(y_t^N-x^*)dt,\label{eq:derdetN3}
\end{align}
where we used in the last inequality that from Assumption \ref{ass:f_convex} it holds $||\nabla f(x)||_*\leq L$ and the Lipschitz continuity of $\nabla\Phi^*$. Using Lemma \ref{lem:conv_md_convex0} (proved in the Appendix Section \ref{sec:lemma_conv_md}), combining inequalities \eqref{eq:derdetN1}, \eqref{eq:derdetN2} and \eqref{eq:derdetN3}, and taking expected values the result follows. 
\end{proof}

A similar result can be obtained for a strongly convex $f$. 
\begin{proposition}[Convergence of ISMD for a strongly convex objective]\label{prop:conv_ismd_strconvex}
Assume that Assumptions \ref{ass:f} -- \ref{ass:lipschitz_bregman} and \ref{ass:f_strconvex} hold.
Then it holds,
\begin{align}
\int_0^T&e^{\mu(t-T)}\mathbb{E}[(f(x_t^i)-f(x^*))]dt \leq \frac{1}{2}e^{-\mu T}D_{\Phi,\mathcal{X}}^2 + \frac{\sigma^2}{2N\mu }(1-e^{-\mu T})|| \Delta \Phi^*||_\infty \\
&+ \int_0^T e^{\mu (t-T)}\frac{L}{\mu}\mathbb{E}\left[|| \tilde z_t^i||_*\right]dt+  \int_0^T e^{\mu (t-T)}\frac{2L+\mu^2}{\mu N}\sum_{i=1}^N \mathbb{E}\left[|| \tilde z_t^i||_*\right]dt .
\end{align}
\end{proposition}
\begin{proof}
Observe that, using strong convexity, we have,
\begin{align}
\frac{1}{N}\sum_{i=1}^N (f(x_t^i)-f(x^*)) \leq& \frac{1}{N}\sum_{i=1}^N\left( -D_f(x^*,x_t^i)+(x_t^i-x^*)^T\nabla f(x_t^i)\right)\\
\leq& \frac{1}{N}\sum_{i=1}^N \left(-\mu D_\Phi(x^*,x_t^i)+(x_t^i-x^*)^T\nabla f(x_t^i)\right)\\
=&\frac{1}{N}\sum_{i=1}^N \left(-\mu D_{\Phi^*}(z_t^i,z^*)+ \mu D_{\Phi^*}(\bar z_t^N,z^*) \right)+\frac{1}{N}\sum_{i=1}^N \nabla f(x_t^i)^T(x_t^i-y_t^N)\\
&+ \frac{1}{N}\sum_{i=1}^N \nabla f(x_t^i)^T(y_t^N-x^*)- \mu D_{\Phi^*}(\bar z_t^N,z^*).
\end{align}
Therefore, 
\begin{align}\label{eq:derdetN4}
\int_0^T e^{\mu t}  \frac{1}{N}\sum_{i=1}^N (f(x_t^i)-f(x^*))dt \leq& \int_0^T e^{\mu t}\frac{1}{N}\sum_{i=1}^N\mu || \tilde z_t^i||_*dt +\int_0^T e^{\mu t}\frac{1}{N}\sum_{i=1}^N \frac{L}{\mu}|| \tilde z_t^i||_*dt \\
&+\frac{1}{2}D_{\Phi,\mathcal{X}}^2+\frac{\sigma^2}{2N}(e^{\mu T}-1)\Delta\Phi^*(\bar z_t^N) + \int_0^Te^{\mu t}\frac{\sigma}{\sqrt{N}}||y_T^N-x^*||_2dW_t,
\end{align}
where we have used the Lipschitz continuity of $\nabla \Phi^*$, the assumption that $||\nabla f(x)||_*\leq L$, the assumption that 
\begin{align}
D_{\Phi^*}(\bar z_t^N,z^*)-D_{\Phi^*}(z_t^i,z^*)\leq ||\bar z_t^N-z_t^i||_*,
\end{align}
 and Lemma \ref{lem:conv_md_strconvex0} (proved in Section \ref{sec:lemma_conv_md} of the Appendix). 
Multiplying by $e^{-\mu T}$, using the assumption that $||\nabla f(x)||_*\leq L$ and the Lipschitz continuity of $\nabla\Phi^*$, combining inequalities \eqref{eq:derdetN1}, \eqref{eq:derdetN2} and \eqref{eq:derdetN4}, and taking expected values the result follows. 
\end{proof}

The deviation from the minimum is upper-bounded by four terms. The first two terms are the standard optimization errors, where we observe that the noise variance is reduced by a factor of $N$. The third and fourth terms are penalties incurred due to each of the particles having different values. These two terms measure the deviation of each individual particle from the particle average. Loosely speaking, if the fluctuation term $|| \tilde z_t^i||_*$  is bounded and small then the interaction will drive the particle system to a state near consensus and the gradient terms will direct the particles towards the minimum. There is furthermore a tradeoff between the interaction and the variance. If the interaction term is bounded and not increasing with $N$, the more particles, the smaller the distance to the optimum, as is witnessed by the term $\frac{\sigma^2}{2N}||\nabla\Phi^*||_\infty$. This implies that the variance of the particles $x_t^i$ is also smaller, and they lie closer around the optimal value. 
Comparing Proposition \ref{prop:conv_ismd_convex} in the convex case and Proposition \ref{prop:conv_ismd_strconvex} in the $\mu$-strongly convex case we observe that the strong convexity of $f$ can speed up convergence, mainly due to the additional factor $\frac{1}{\mu}$. We remark that the results presented in this section thus show that the expected value of the distance of the objective function evaluated at the time average to the objective function evaluated in the optimum decreases as $N$ increases. In addition to this, as we will present in our numerical results later in Section \ref{sec:num} that both the variance of the cost function as well as the variance of the samples are smaller when using interaction. 

\subsubsection{A comparison with averaging}
Consider the case of the convex objective. The interaction in the previous section is introduced in order to obtain a way to control the effects of noise. In particular we observed that if the term $|| \tilde z_t^i||_*$ was bounded and not an increasing function of $N$, the variance could be reduced by having interacting particles. We can compare the setting with interaction to one in which we simply average the trajectories of the particles.

If the following holds,
\begin{align}\label{eq:asslin}
\frac{1}{N}\sum_{i=1}^N\nabla f(x_t^i) = \frac{1}{N}\sum_{i=1}^N\nabla f\circ\nabla\Phi^*(z_t^i) = \nabla f\circ\nabla\Phi^*\left(\frac{1}{N}\sum_{i=1}^N z_t^i\right) := \nabla f(y_t^N),
\end{align}
then we have from Lemma \ref{lem:conv_md_convex0} (in Appendix \ref{sec:lemma_conv_md}) and the convexity of $f$ that
\begin{align}
\mathbb{E}\left[\frac{1}{T}\int_0^T (f(y_t^N) - f(x^*)) dt\right]\leq \frac{1}{2T}D_{\Phi,\mathcal{X}}^2 + \frac{\sigma^2}{2N}||\Delta \Phi^*(\bar z_t^N)||_\infty.
\end{align}
Notice that the Assumption in \eqref{eq:asslin} holds for a function of the form $f(x) = a^Tx+b+\Phi(x)$ which is the particular loss function considered in \cite{raginsky12}. 
Thus in the setup of this kind structure in the loss function and gradient, averaging the particles can decrease the effect of the noise. 

In a general setting where the loss function is nonlinear, one can decrease the effects of noise by averaging the particles trajectories. Combining \eqref{eq:derdetN2}, \eqref{eq:derdetN3} and Lemma \ref{lem:conv_md_convex0} we have,
\begin{align}
\mathbb{E}\left[\frac{1}{T}\int_0^T (f(y_t^N) - f(x^*)) dt\right]\leq& \frac{1}{2T}D_{\Phi,\mathcal{X}}^2 + \frac{\sigma^2}{2N}||\Delta \Phi^*(\bar z_t^N)||_\infty + \int_0^T\frac{2L}{\alpha TN}\sum_{i=1}^N|| \tilde z_t^i||_*dt.
\end{align}
In the nonlinear case the average of the gradients is not the gradient of the average and therefore the additional term representing the deviations of each individual particle from the particle average plays a role. 

\subsection{Bounding the fluctuation term}\label{sec:bound_fluct}
In this section we present a bound on the fluctuation term. When the fluctuation term is sufficiently small all particles have approximately the same value; we refer to this as the particles having achieved \emph{consensus}. 
The dynamics of the fluctuation term in a vectorized form is given by,
\begin{align}
d\mathbf{\tilde z}_t = \left[\left(\frac{1}{N}\mathbf{1}_N\mathbf{1}_N^T\otimes I_d-I_{dN}\right)\nabla \mathbf{V}(\mathbf{z}_t)- \mathcal{L} \mathbf{z}_t\right]dt + \sigma \sqrt{\frac{(N-1)}{N}}d\mathbf{B}_t,
\end{align}
where we have used that $\left(I_{dN}-\frac{1}{N}\mathbf{1}_N\mathbf{1}_N^T\otimes I_d\right)d\mathbf{B}_t\overset{d}{=}\sqrt{\frac{(N-1)}{N}}d\mathbf{B}_t$.
\begin{proposition}[Bounding the fluctuation term: a general result for the convex case]\label{prop:fluctuation_stochastic}
Assume that Assumptions \ref{ass:f} -- \ref{ass:mirror2} and \ref{ass:V_convex} hold. Then,
\begin{align}
\frac{1}{TN}\int_0^T\sum_{i=1}^N\mathbb{E}\left[\mathcal{V}(z_t^i)-\mathcal{V}(\bar z_t^N)\right]dt \leq \frac{L}{2TN}\sum_{i=1}^N ||\tilde{z}_0^i||^2_2 + \frac{1}{2}d\sigma^2\frac{(N-1)}{N}.
\end{align}
\end{proposition}
\begin{proof}
Define the Lyapunov function,
\begin{align}
V_t(\mathbf{\tilde z}) := \frac{1}{2}\mathbf{\tilde z}^T\mathbf{\tilde z}.
\end{align} 
By It\^o's lemma it then follows that, 
\begin{align}
dV_t&= \mathbf{\tilde z}^T \left[\left(\frac{1}{N}\mathbf{1}\mathbf{1}^T\otimes I_d-I_{dN}\right)\nabla \mathbf{V}(\mathbf{z}_t)-\mathcal{L} \mathbf{z}_t\right]dt+dN\frac{1}{2}\sigma^2\frac{(N-1)}{N}dt+\mathbf{\tilde z}^T\sigma \sqrt{\frac{(N-1)}{N}}d\mathbf{B}_t.
\end{align}
Define $\mathbf{\bar z}_t^N = 1_{N}\otimes z_t^N$. We can rewrite the drift term of $dV_t$ as,
\begin{align}
\mu_t:=&\mathbf{\tilde z}^T\left(\frac{1}{N}\mathbf{1}\mathbf{1}^T\otimes I_d\right)\nabla \mathbf{V}(\mathbf{z}_t)-\mathbf{\tilde z}^T\nabla \mathbf{V}(\mathbf{z}_t)-\mathbf{\tilde z}^T\mathcal{L} \mathbf{z}_t+dN\frac{1}{2}\sigma^2\frac{(N-1)}{N}. 
\end{align}
Note that, 
\begin{align}
&\mathbf{\tilde z}^T\left(\frac{1}{N}\mathbf{1}\mathbf{1}^T\otimes I_d\right)\nabla \mathbf{V}(\mathbf{z}_t)=0.
\end{align}
By Assumption \ref{ass:f_convex}, 
\begin{align}
-\mathbf{\tilde z}^T\nabla \mathbf{V}(\mathbf{z}_t)\leq \sum_{i=1}^N \left(\mathcal{V}(\bar z_t^N)-\mathcal{V}( z_t^i)\right).
\end{align}
Furthermore, by the properties of the graph Laplacian, it holds $\mathcal{L}1_{dN}=0$, so that,
\begin{align}
-\mathbf{\tilde z}^T_t \mathcal{L}\mathbf{z}_t = -\mathbf{\tilde z}^T_t\mathcal{L}\mathbf{\tilde z}_t\leq -\underline{\lambda} \mathbf{\tilde z}^T_t\mathbf{\tilde z}_t\leq 0.
\end{align}
Then, 
\begin{align}
V_T\leq V_0 + \int_0^T \sum_{i=1}^N\left(\mathcal{V}(\bar z_t^N)-\mathcal{V}( z_t^i)\right) dt + T dN\frac{1}{2}\sigma^2\frac{(N-1)}{N} + \int_0^T\sigma \sqrt{\frac{(N-1)}{N}}||\mathbf{\tilde z}_t||_2 d\mathbf{B}_t,
\end{align}
where we have additionally used that $\mathbf{\tilde z}^Td\mathbf{B}_t \overset{d}{=} ||\mathbf{\tilde z}||_2 d\mathbf{B}_t$.
Rearranging, using the property of the It\^o integral, and using $V_T\geq0$ we find,
\begin{align}
\frac{1}{TN}\int_0^T\sum_{i=1}^N \mathbb{E}\left[\mathcal{V}( z_t^i)-\mathcal{V}(\bar z_t^N) \right] dt \leq \frac{1}{T}V_0 + d\frac{1}{2}\sigma^2\frac{(N-1)}{N}.
\end{align}
\end{proof}

We note that in the above result we used the convexity of $\mathcal{V}$ but we did not exploit the interaction between the particles. From the above statement we observe that in a deterministic setting, i.e. if $\sigma=0$, consensus can be achieved even without interaction. Each particle is driven towards the optimum $x^*$ by the gradient term. Consensus is in this case achieved exactly at optimality. In the stochastic setting exact consensus can no longer be achieved and the distance to optimality remains bounded by a function of the noise.
\begin{proposition}[Bounding the fluctuation term: the strongly convex case]\label{prop:fluctuation_stochastic_strconvex}
Let Assumptions \ref{ass:f} -- \ref{ass:mirror2} and \ref{ass:V_strconvex} hold. Then,
\begin{align}
\mathbb{E}\left[\frac{1}{N}\sum_{i=1}^N||\tilde{z}_t^i||_*^2\right] \leq \frac{K}{N}e^{-(\kappa+\underline{\lambda})t} \sum_{i=1}^N ||\tilde{z}_0^i||^2_2 + \frac{dK}{(\kappa+\underline{\lambda})}\sigma^2\frac{(N-1)}{N}\left(1-e^{-(\kappa+\underline{\lambda})t}\right).
\end{align}
\end{proposition}
\begin{proof}
Define again the Lyapunov function,
\begin{align}
V_t(\mathbf{\tilde z}) := \frac{1}{2}\mathbf{\tilde z}^T\mathbf{\tilde z}.
\end{align} 
It is clear that $V_t=0$ if $z_t^i=z_t^j$, or in other words consensus has been achieved and the fluctuation term is zero.
By It\^o's lemma it then follows that, 
\begin{align}
dV_t&= \mathbf{\tilde z}^T \left[\left(\frac{1}{N}\mathbf{1}\mathbf{1}^T\otimes I_d-I_{dN}\right)\nabla \mathbf{V}(\mathbf{z}_t)-\mathcal{L} \mathbf{z}_t\right]dt+dN\frac{1}{2}\sigma^2\frac{(N-1)}{N}dt+\mathbf{\tilde z}^T\sigma \sqrt{\frac{(N-1)}{N}}d\mathbf{B}_t.
\end{align}
Define $\mathbf{\bar z}_t^N = 1_{N}\otimes z_t^N$. We can rewrite the drift term of $dV_t$ as,
\begin{align}
\mu_t:=&\mathbf{\tilde z}^T\left(\frac{1}{N}\mathbf{1}\mathbf{1}^T\otimes I_d\right)\nabla \mathbf{V}(\mathbf{z}_t)+\mathbf{\tilde z}^T(\nabla \mathbf{V}(\mathbf{\bar z}_t^N)-\nabla \mathbf{V}(\mathbf{z}_t))-\mathbf{\tilde z}^T\nabla\mathbf{V}(\mathbf{\bar z}_t^N)\\
&-\mathbf{\tilde z}^T\mathcal{L} \mathbf{z}_t+dN\frac{1}{2}\sigma^2\frac{(N-1)}{N}. 
\end{align}
Note that, since $\sum_{i=1}^N(z_t^i-\bar z_t^N)=0$,
\begin{align}
&\mathbf{\tilde z}^T\left(\frac{1}{N}\mathbf{1}\mathbf{1}^T\otimes I_d\right)\nabla \mathbf{V}(\mathbf{z}_t)=0,\;\;\;\;\mathbf{\tilde z}^T\nabla\mathbf{V}(\mathbf{\bar z}_t^N)=0.
\end{align}
By Assumption \ref{ass:V_strconvex}, 
\begin{align}
\mathbf{\tilde z}^T(\nabla \mathbf{V}(\mathbf{\bar z}_t^N)-\nabla \mathbf{V}(\mathbf{z}_t))\leq -\kappa \mathbf{\tilde z}^T\mathbf{\tilde z}.
\end{align}
Furthermore, by the properties of the Laplacian, it holds $\mathcal{L}1_{dN}=0$, so that,
\begin{align}
\mathcal{L}\mathbf{z}_t = \mathcal{L}\mathbf{\tilde z}_t.
\end{align}
Therefore,
\begin{align}
\mu_t &\leq -\kappa\mathbf{\tilde z}^T\mathbf{\tilde z}-\mathbf{\tilde z}^T\mathcal{L}\mathbf{\tilde z}+dN\frac{1}{2}\sigma^2\frac{(N-1)}{N}\\
&\leq- (\kappa+\underline{\lambda})\mathbf{\tilde z}^T\mathbf{\tilde z}+dN\frac{1}{2}\sigma^2\frac{(N-1)}{N}.
\end{align}
Integrating the expression and using the above bound we obtain,  
\begin{align}
V_t \leq e^{-(\kappa+\underline{\lambda})t}V_0 + \frac{1}{2}\int_0^t e^{-(\kappa+\underline{\lambda})(t-s)}dN\sigma^2\frac{(N-1)}{N}ds +\int_0^t e^{-(\kappa+\underline{\lambda})(t-s)}\sigma \sqrt{\frac{(N-1)}{N}}||\mathbf{\tilde z}_s||_2 d\mathbf{B}_s,
\end{align}
where we have additionally used that $\mathbf{\tilde z}^Td\mathbf{B}_t \overset{d}{=} ||\mathbf{\tilde z}_t||_2 d\mathbf{B}_t$.
Then, taking expectations we have:
\begin{align}
\mathbb{E}[V_t] \leq e^{-(\kappa+\underline{\lambda})t} V_0 + \frac{1}{2(\kappa+\underline{\lambda})}dN\sigma^2\frac{(N-1)}{N}\left(1-e^{-(\kappa+\underline{\lambda})t}\right).
\end{align}
The statement follows from norm equivalence, i.e. $||\cdot||_*^2\leq K||\cdot||_2^2$. 
\end{proof}

\begin{remark}[Constant in the norm equivalence]
If the Euclidean norm is used above in Assumptions \ref{ass:f} -- \ref{ass:mirror1}, then $K=1$. In certain cases however, $K$ can be dimension-dependent. 
\end{remark}

We furthermore note, by the bound on $\mathbb{E}\left[\frac{1}{N}\sum_{i=1}^N||\tilde{z}_t^i||_*^2\right]$ implies a bound on $\mathbb{E}\left[\frac{1}{N}\sum_{i=1}^N||\tilde{z}_t^i||_*\right]$ with the dependence on $N$ remaining the same. This is formalized in the following lemma. 
\begin{lemma}
It holds for some constant $K$, 
\begin{align}
\mathbb{E}\left[\frac{1}{N}\sum_{i=1}^N||\tilde{z}_t^i||_*\right]\leq \left(\frac{K}{N}e^{-(\kappa+\underline{\lambda})t} \sum_{i=1}^N ||\tilde{z}_0^i||^2_2 + \frac{dK}{(\kappa+\underline{\lambda})}\sigma^2\frac{(N-1)}{N}\left(1-e^{-(\kappa+\underline{\lambda})t}\right)\right)^\frac{1}{2}.
\end{align}
\end{lemma}
\begin{proof}
Let $||\cdot||$ be an arbitrary norm. Observe that by H\"older's inequality $\mathbb{E}\left[|x|\right]\leq \left(\mathbb{E}\left[|x|^2\right]\right)^{\frac{1}{2}}$. Let $x=\sum_{i=1}^N||z_t^i||$. Observe that,
\begin{align}
|x|^2=\sum_{i=1}^N||z_t^i||^2+\sum_{i=1}^N\sum_{j\neq i}^N||z_t^i||\; ||z_t^j||&\leq\sum_{i=1}^N||z_t^i||^2+\frac{1}{2}\sum_{i=1}^N\sum_{j\neq i}^N\left( ||z_t^i||^2+||z_t^j||^2\right)\\
\leq&(1+N)\sum_{i=1}^N||z_t^i||^2,
\end{align}
where we have applied Young's inequality in the second inequality. We obtain,
\begin{align}
\mathbb{E}\left[\frac{1}{N}\sum_{i=1}^N||z_t^i||\right]\leq \mathbb{E}\left[\frac{1+N}{N^2}\sum_{i=1}^N||z_t^i||^2\right]^{\frac{1}{2}}.
\end{align}
The statement follows using Proposition \ref{prop:fluctuation_stochastic_strconvex}. 

\end{proof}

The above statements show that for a strongly convex objective approximate consensus can be achieved for a sufficiently high $t$ and for a sufficiently large $\kappa+\underline{\lambda}$. In the strongly convex case approximate consensus can be achieved even in the case of no interaction if $\kappa$ is sufficiently large. For the convex case we have $\kappa=0$. In this case approximate consensus can be achieved if $\underline{\lambda}$ is sufficiently high. For a convex objective without interaction we would be left with the noise term similar to the result in Proposition \ref{prop:fluctuation_stochastic}. Without interaction the strong convexity of the objective thus determines in how much the effect of noise can be reduced. With interaction the interaction type itself, i.e. the $\underline{\lambda}$, also plays a role and for a sufficiently high $\underline{\lambda}$ the effects of noise on consensus can be minimized. We remark furthermore that we bound the dual norm by the $L_2$ norm using norm equivalence; the downside of this is that an additional dimensionality-dependence can be obtained. This could be mitigated by working with directly the dual norm as in \cite{raginsky12} but in this case the obtained result is rather limited to a specific set of objective functions; alternatively one can work in the primal space directly. The latter idea will be addressed in future work. 

\subsubsection{Extensions to distributed optimization}
A common setting is to consider minimizing $f(x)=\sum_{i=1}^N f_i(x)$ and each particle corresponds to a computing worker or processor having access only to $f_i$ and thus $\nabla f_i\circ \nabla\Phi^*(z_t^i)$ would be used in \eqref{eq:cmdpart} (instead of $\nabla f\circ \nabla\Phi^*(z_t^i)$). The proof of Proposition \ref{prop:fluctuation_stochastic} and \ref{prop:fluctuation_stochastic_strconvex} is based the fact that each particle has access to the full objective function $f$. As a result terms that would be relevant in a distributed setup related to variability of each $f_i$ vanish. 
Clearly, this results in a better bound and thus a better convergence result, but the current results can be extended for the distributed case by increasing the interaction strength and multiplying $A$ in \eqref{eq:cmdpart} by a sufficiently high constant to ensure consensus arises. The details are left for future work.

\subsubsection{On the tradeoff between noise reduction and interaction}
From the results in Section \ref{sec:gen_bound_ismd} we observed that for interacting particles the convergence can be improved since the effect of noise in the term $\frac{1}{N}\sigma ||\Delta\Phi(z_t^i)||_\infty$ is reduced by a factor $1/N$. We however remark that this will only be achieved if the term $|| \tilde z_t^i||_*$ is bounded and non-increasing with $N$. As we showed in Section \ref{sec:bound_fluct} we can present a bound on the interaction term. In particular, we showed that the fluctuation is bounded and non-increasing with $N$. The fluctuation is bounded in the $L_2$ norm, and a noise term does remain; as discussed, this noise term can be controlled by the strong convexity of the objective or by imposing an interaction strength, which means replacing $A$ with $\vartheta A$ for $\vartheta>1$. In other words, when $\underline{\lambda}+\kappa$ is sufficiently high the fluctuation term is sufficiently small. Specifically, in the strongly convex case for a large $N,T$ the fluctuation term is bounded by $\frac{dK}{(\kappa+\underline{\lambda})}\sigma^2$ for some constant $K$. We conclude this section by saying that as long as the decrease in the value of $\frac{1}{N}\sigma ||\Delta\Phi(z_t^i)||_\infty$ is larger than the increase in the value of $|| \tilde z_t^i||_*$, interaction with $N$ particles achieves a \emph{closer} convergence to the optimum. In other words, the amount of particles can be seen as an alternative to a decreasing learning rate or vanishing noise variance (e.g. the latter can be achieved by increasing the mini-batch size).

\section{Understanding performance from convergence to stationarity}\label{sec:sample}

We proceed with exploiting tools from the analysis of SDEs and in particular the rate of convergence to their invariant distribution. Studying the convergence of the $Law(z_t)$ or $Law(\mathbf{z}_{t})$ is a cornerstone in the analysis of sampling schemes	 
and can also provide valuable insights into the optimization problems; see\cite{raginsky2017non,shi2020learning} for recent works in this direction. In particular, we will discuss the convergence of \eqref{eq:smd} or \eqref{eq:cmdpart} to the corresponding stationary
distributions. Note that in the spirit of dual (or Nesterov) averaging
the SDE in \eqref{eq:cmdpart} is fairly standard as it evolves
only in the mirror space. In the analysis so far we were investigating
how close a time average $\frac{1}{T}\int_{0}^{t}f(x_{t}^{i})dt$ (or
$\frac{1}{T}\int_{0}^{t}f\circ\nabla\Phi^{*}(z_{t}^{i})dt$) is to $f(x^{*})$
using a Lyapunov method based on Bregman divergence. We will complement these results with rates of convergence
based on logarithmic Sobolev inequalities based on the celebrated
Bakry-Emery approach, see \cite{bakry1997sobolev}, \cite{bakry2013analysis} for details. 

\subsection{Stationary distributions and particle correlations}


We first consider the case when particles do not interact. One approach for finding $x^*$ is to run $N$ independent copies of \eqref{eq:smd} for a long time and then select best particle $i^*$ as $\arg\min_{i\in \{ 1,\ldots,N\} }f(x^i_t)$. 
This intuitive approach is based on
$\min f$ being equivalent to $\max\exp(-\frac{2}{\sigma^2}\mathcal{V})$ and on the convergence of $\frac{dLaw(z_t)}{dz}(z)$ to $\frac{1}{Z}\exp(-\frac{2}{\sigma^2}\mathcal{V})$. 
When stacking the particles together in $\mathbf{z}$, then one can postpone choosing the best particle and 
considers $\min_{\mathbf{x}} \sum_{i=1}^N f(x^i)$ or $\min_{\mathbf{z}} \sum_{i=1}^N V(z^i)$ instead, with each particle $i$ here sharing the same properties and dynamics.

To make this a bit clearer for our setup, denote with $\eta_t=Law(\mathbf{z}_t)$ the law of $N$ independent particles each following \eqref{eq:smd}; or equivalently each $z_t^i$ evolving as \eqref{eq:cmdpart} with $A=0$. It is well-known (see e.g. \cite{pavliotis14book}) that under appropriate assumptions on the objective $\mathcal{V}$, the stationary distribution of this SDE is given by,
\begin{align}
\eta_{\infty}\left(d\mathbf{z}\right)=\frac{1}{Z}\exp\left(-\frac{2}{\sigma^{2}}\sum_{i=1}^N \mathcal{V}(z^i)\right)d\mathbf{z},\label{eq:eta_iid}
\end{align}
with $Z$ being a finite normalizing constant.
It is clear that finding the mode of $\eta_{\infty}$ is equivalent to solving the following optimization problem: 
\begin{align}\label{eq:opt_sum_V}
\mathbf{z^*}=\arg\min_{\mathbf{z}\in\mathbb{R}^{dN}}\sum_{i=1}^N \mathcal{V}(z^i).
\end{align}

An important observation here is that independence of each $i$ is not crucial and one can modify the cost $\mathcal{V}(z^i)$ by adding terms that do not affect the minimizer. This is true when using instead the following cost function for $\mathbf{z}$:
\begin{align}\label{eq:defw}
\mathcal{W}(\mathbf{z})=\sum_{i=1}^{N}\mathcal{V}(z^{i})+\frac{1}{2}\sum_{i=1}^{N}\sum_{j=1}^{N}A_{ij}(z^{i}-z^{j})^{2}.
\end{align}
Note that the quadratic term due to interaction terms acts as a regularizer,
which aims to impose consensus i.e. $z^{i}=\bar{z}^N$.  We can make the this claims more precise in the following lemma. 

\begin{lemma}[Interaction preserves the minimum]\label{lem:minimum_preserved}
Let $z^*:=\arg\min_{z}\mathcal{V}(z)$, and let $\mathcal{W}(\mathbf{z})$ be as defined in \eqref{eq:defw}. Then $\mathbf{z^*}={({z^*}^T,\ldots,{z^*}^T)}^T$ is a minimizer of $\mathcal{W}(\mathbf{z})$. 
\end{lemma}
\begin{proof}
From the definition of $z^*$  and the optimality conditions we have $\nabla \mathcal{V}(z^*)=0$ and that there exists $\beta>0$ such that $\textnormal{Hess}(\mathcal{V})(z^*)\succeq \beta I_d$. 
We have,
\begin{align}
(\nabla \mathcal{W}(\mathbf{z}))_i = \nabla\mathcal{V}(z^{i}) + \sum_{j\neq i} A_{ij}(z_i-z_j).
\end{align}
Clearly, when $z^i=z^*$ for every $i$ we then have $\nabla \mathcal{W}(\mathbf{z})=0$.
Furthermore:
\begin{align}\label{eq:hess_w_ii}
(\textnormal{Hess}(\mathcal{W}(\mathbf{z}))_{ii} = \textnormal{Hess}(\mathcal{V}(z^{i}))+ \sum_{j\neq i} A_{ij},
\end{align}
and
\begin{align}\label{eq:hess_w_ij}
(\textnormal{Hess}(\mathcal{W}(\mathbf{z}))_{ij} = - A_{ij}.
\end{align}
This implies
\begin{align}
(\textnormal{Hess}(\mathcal{W}(\mathbf{z^*}))_{ii} \geq \beta  - \sum_{j\neq i} (\textnormal{Hess}(\mathcal{W}(\mathbf{z^*}))_{ij},
\end{align}
so that $\textnormal{Hess}(\mathcal{W})(\mathbf{z}^*)\succeq \beta I_{dN}$ and $\mathbf{z^*}$ minimizes $\mathcal{W}$ since $\beta>0$.
\end{proof}

We will now extend the previous discussion on invariant distributions and $\mathbf{z^*}$ for ISMD and \eqref{eq:cmdpart} with $A$ being nonzero. We will denote the distribution of the joint particle system and each marginal as follows:
\begin{align}
\eta_{t}^{N}=Law(\mathbf{z}_{t})\qquad\text{ and }\qquad \eta_{t}^{i,N}=Law(z_{t}^{i}).
\end{align}
Recall $\mathbf{z}_t$ evolves as in \eqref{eq:cmdpart_vec}, 
whose the invariant distribution is given as
\begin{align}
\eta_{\infty}^{N}\left(d\mathbf{z}\right)=\frac{1}{Z_{N}}\exp\left(-\frac{2}{\sigma^{2}}\left(\sum_{i=1}^{N}\mathcal{V}(z^{i})+\frac{1}{2}\mathbf{z}^{T}\mathcal{L}\mathbf{z}\right)\right)d\mathbf{z},\label{eq:eta_int},
\end{align}
where $Z_{N}$ is the normalization constant. 
Similar as in the non-interacting setting, finding
the mode of $\eta_{\infty}^{N}$ is equivalent to solving the following
optimization problem: 
\begin{align}
\mathbf{\tilde z}=\arg\min_{z\in\mathbb{R}^{dN}}\mathcal{W}(z),
\end{align}
which we saw earlier in Lemma \ref{lem:minimum_preserved} is equivalent to solving \eqref{eq:opt_sum_V}. 

Despite introducing correlations between the particles through $\mathcal{L}$, the mode of the invariant distribution of the interacting particle system is exactly the minimizer of the objective $\mathcal{V}$ and is achieved when all particles are at consensus. The presence of the noise means that exact consensus or synchronization of all particles cannot be achieved.
Similar to SMD stochasticity implies that at stationarity the $z_t^i$ will behave as samples from the invariant distribution and thus oscillating together around the optimum $z^*$. In this context it is convenient to use convergence results for $\eta_{t}^{i,N}$ for deriving performance bounds for optimization or assessing 
the level of consensus based on differences between $\eta_{t}^{i,N}$ and $ \eta_{t}^{j,N}$.


In addition to establishing positive curvature at $\mathbf{z^*}$, using diagonal dominance in the final steps in the proof of Lemma \ref{lem:minimum_preserved} implies positive definiteness for $\mathcal{W}$. This is not surprising given the sum of quadratics in the interaction term acts as a convex regularizer. We summarize this in the corollary below: 
\begin{corollary}\label{cor:convex}
Let Assumptions \ref{ass:f} -- \ref{ass:mirror2} and \ref{ass:V_strconvex} hold. Then $\mathcal{W}$ is strongly convex.
\end{corollary}
\begin{proof}
 Assumption \ref{ass:V_strconvex} implies that $\textnormal{Hess}(\mathcal{V})(z)\succeq \kappa I_d$ for some $\kappa> 0$. Note that \eqref{eq:hess_w_ii}-\eqref{eq:hess_w_ij} hold for any $\mathbf{z}$ and 
\begin{align}
\textnormal{Hess}(\mathcal{W})=\textnormal{Hess}\left(\sum_{i=1}^{N}\mathcal{V}(z^{i})\right)+\frac{1}{2}\textnormal{Hess}(\mathbf{z}^{T}\mathcal{L}\mathbf{z})\succeq\kappa  I_{dN}.
\end{align} 
\end{proof}
The result is somewhat negative in that to obtain strict convexity of $\mathcal{W}$ one requires strong convexity of $\mathcal{V}$, that is adding $z^{T}\mathcal{L}z$ is not sufficient to obtain strong convexity when $\mathcal{V}$ is only convex. This is particularly relevant later when we will apply Bakry-Emery theory and derive log Sobolev inequalities. 

\subsection{Log-Sobolev inequalities for the particle system}\label{sec:log_sob_N}

We proceed to present the exponential convergence properties for $\eta_t^N$. Our first result in this section is to apply Bakry-Emery theory to derive the Log-Sobolev inequality for the particle system. Similar to the result in Lemma 3.5, Corollary 3.7 in \cite{malrieu01} we have the following result:
\begin{proposition}[Convergence of ISMD using Bakry-Emery theory]\label{prop:logsobolev_particle}
Let Assumptions \ref{ass:f} -- \ref{ass:mirror2} and \ref{ass:V_convex} hold (w.r.t the Euclidean norm). 
Assume also $\eta_{0}^{N}$ satisfies a log-Sobolev inequality with
constant $C_{0}.$ Then $\eta_{t}^{N}$ satisfies a log-Sobolev inequality for any smooth function $f$
\begin{align}\label{eq:lsi_particles}
\textnormal{Ent}_{\eta_{t}^{N}}(f^2)\leq C_{t}\mathbb{E}_{\eta_{t}^{N}}\left[\left|\nabla f\right|^{2}\right]
\end{align}
with constant $C_{t}=\frac{2}{\rho}\left(1-e^{-\rho t}\right)+C_{0}e^{-\rho t}$ 
where $\rho=\frac{\sigma^{2}}{2}\kappa$.
If in addition Assumption \ref{ass:V_strconvex} holds, we have 
\begin{equation}
H(\eta_{t}^{N}|\eta_{\infty}^{N})\leq Ke^{-2\rho t}\quad\text{and}\quad W_{2}(\eta_{t}^{N},\eta_{\infty}^{N})\leq\sqrt{\frac{K}{2}}e^{-\rho t},\label{eq:KL_bound_eta_t}
\end{equation}
and 
\begin{align}\label{eq:erm}
|\mathbb{E}_{\eta^{N}_t}\mathcal{W}-\mathcal{W}(\mathbf{z^*})| \leq \sqrt{\frac{K}{2}}e^{-\rho t}+\frac{\sigma^{2}}{2}\left(\frac{2dN}{\rho}-\frac{1}{2}\log\left(\frac{\sigma^{2}}{2L_{N}}\right)\right),
\end{align}
with $L_N = \frac{L}{\mu}+||\mathcal{L}||$ where $||\mathcal{L}||$ is any matrix norm on $\mathcal{L}$. 
\end{proposition}


\begin{proof}
For the proof we will rely on the Bakry-Emery theory. To state the Bakry-Emery criterion for \eqref{eq:cmdpart_vec}, we need to define the following differential operators:
\[
\mathfrak{L}(f)=-\nabla\mathcal{W}^T\nabla f +\frac{\sigma^2}{2} \textnormal{Hess} (f), \Gamma(f)=\frac{1}{2}\left(\mathfrak{L}(fg)- f\mathfrak{L}(g)- g\mathfrak{L}(f)\right), \Gamma_2(f)=\frac{1}{2}\left(\mathfrak{L}(\Gamma(f,f))-2\Gamma((\mathfrak{L}(f),f) \right).
\]
The Bakry-Emery criterion requires to verify $\Gamma_2(f)\succeq \rho\Gamma(f,f)$. We have:
\begin{align}
\Gamma_2(f)-\rho\Gamma(f,f)=\Big( \frac{\sigma^2}{2}\Big)^2\|\textnormal{Hess}(f)\|^2+\frac{\sigma^2}{2}\|\nabla f \|^2  ( \frac{\sigma^2}{2} \textnormal{Hess} (\mathcal{W})- \rho I ),
\end{align}
so need to verify 
\begin{align}
\textnormal{Hess}(\mathcal{W})\succeq \frac{2}{\sigma^2}\rho I_{dN},
\end{align}
that holds for $\rho=\frac{\sigma^{2}}{2}\kappa$ from Assumption \ref{ass:V_convex} or  Assumption \ref{ass:V_strconvex}, Corollary \ref{cor:convex} and $\mathcal{L}\mathbf{1}=0$. 
For the second part we explicitly require $\rho>0$, which follows straightforwardly from Assumption \ref{ass:V_strconvex}.
This gives directly $C_t$ and we note that the second term in $C_t$ is due to the random initialization (e.g. see \cite[Corollary 3.7]{malrieu01}). When Assumption \ref{ass:V_strconvex} holds and $\rho>0$ we have 
\begin{equation}
H(\eta_{t}^{N}|\eta_{\infty}^{N})\leq Ke^{-2\rho t},\label{eq:conv_H}
\end{equation}
for $K=H(\eta_0^N|\eta_\infty^N)$; see Theorem \ref{thm:bakry-emery} in Appendix \ref{sec:log_sob_primer}.  Theorem \ref{thm:lsi} gives
$W_{2}(\eta_{t}^{N},\eta_{\infty}^{N})\leq\sqrt{\frac{K}{2}}e^{-\rho t}$.

For the mean-mode result, we first need to establish $\eta^N_\infty\in\mathcal{P}_2(\mathbb{R}^{dN})$. This can be shown using Theorem \ref{thm:lsi} in Appendix \ref{appx:erm} with the log Sobolev constant $C_\infty=\frac{2}{\rho}$ implies
\[
\eta^{N}_\infty\left(e^{sf}\right)\leq e^{s\eta^{N}_\infty(f)+C_\infty s^{2}},
\]
for any 1-Lipschitz function $f$, so consider $f(\mathbf{z})=\mathbf{z}$ and the moment generating function to get the following moment bounds
for $s=0$
\[
\mathbb{E}_{\eta^{N}_\infty}\left[\mathbf{z}\mathbf{z}^{T}\right]\preceq\left(2C_{\infty}I_{dN}+E_{\eta^{N}_\infty}\left[\mathbf{z}\right]E_{\eta^{N}_\infty}\left[\mathbf{z}\right]^{T}\right),
\]
so moments can be bound using exponential integrability and $\eta^N_\infty$ has finite second moment. 
Combining this with 
\begin{equation}
\log\det\left(\Sigma^N\right)\leq Tr\left(\Sigma^N-I\right)\label{eq:log_det_ineq}
\end{equation}
gives 
\[
\log\det\left(\Sigma^N\right)\leq dN\left(2C_{\infty}-1\right).
\]
We can then gather Corollary \ref{cov_ips_bound} and Proposition \ref{prop:logZ_lb} in Appendix \ref{appx:erm} we get 
\begin{align*}
\mathbb{E}_{\eta^{N}_\infty}\mathcal{W}-\mathcal{W}(\mathbf{z^*}) & \leq\frac{\sigma^{2}}{2}\left(\frac{dN}{2}+\frac{1}{2}\log\det\left(\Sigma^{N}\right)-\frac{1}{2}\log\left(\frac{\sigma^{2}}{2L_{N}}\right)\right)\\
 & \leq\frac{\sigma^{2}}{2}\left(\frac{2dN}{\rho}-\frac{1}{2}\log\left(\frac{\sigma^{2}}{2L_{N}}\right)\right)
\end{align*}
The statement then follows by noting that,
\begin{align}
|\mathbb{E}_{\eta^{N}_t}\mathcal{W}(\mathbf{z}_t)-\mathcal{W}(\mathbf{z^*})|\leq |\mathbb{E}_{\eta^{N}_t}\mathcal{W}(\mathbf{z}_t)-\mathbb{E}_{\eta^{N}_\infty}\mathcal{W}(\mathbf{z}_t)| + |\mathbb{E}_{\eta^{N}_\infty}\mathcal{W}(\mathbf{z}_t)-\mathcal{W}(\mathbf{z^*})|,
\end{align}
and using the convergence rate from the result derived above in \eqref{eq:KL_bound_eta_t} for the first term and the mean-mode bound for the second. 
\end{proof}


This result shows that the rate of convergence to the invariant measure is \emph{exponential} in time with factor $\rho$. Using the relationship between the mode of the invariant measure and the minimizer of the objective function, this result implies that with exponential convergence the samples $z_t^i$ will oscillate around the optimum $z^*$ for a strongly convex objective. 
Note that the discrepancy between the mean and the mode at stationarity seen in \eqref{eq:erm} means this oscillation is not centred around $z^*$. Equation \eqref{eq:erm} also
shows that when the width of $\eta_{\infty}^{N}$
is small, i.e. $\sigma^{2}$ is low, then $\mathbf{z}_{t}$
will lie closer to the minimizer.
To identify $z^*$ in practice, one needs to look at the histogram of $\mathbf{z}^i_t$ (w.r.t $i$) or the occupation measure of $\mathcal{W}(\mathbf{z}_{t})$ after some burn-in time $t_0$ or a combination of the two. The amount of required burn-in can be lower when fast convergence to stationarity occurs.
From the expression for $\rho$ we can see that the rate of convergence $\rho$ increases with the noise variance and the strong convexity coefficient of the objective function.
Finally, the dependence on $N$ arrives from dimensionality dependence of certain bounds involving Gaussian integrals; see Proposition \ref{prop:logZ_lb} in Appendix \ref{appx:erm}.

At the level of the marginal $\eta_{\infty}^{i,N}$ a similar bound to \eqref{eq:erm} can be obtained, where $dN$ is replaced with $d$. 
Convergence of $\eta_{t}^{N}$ to $\eta_{\infty}^{N}$ implies
convergence of the marginals so one can deduce a similar result to Proposition \ref{prop:logsobolev_particle} for the marginals and a similar concentration inequality for $\mathcal{V}(z_t^i)$ showing that the dimension dependence of the optimization is with $d$ and not $dN$.  The details are omitted as the proof is very similar to Proposition \ref{prop:logsobolev_particle}. 



\subsubsection{On the tradeoff between noise reduction and interaction: a sampling perspective}
One could attempt to compare Proposition \ref{prop:logsobolev_particle} with the results obtained in Section \ref{sec:gen_bound_ismd}, although they are quite different in nature. In Propositions \ref{prop:conv_ismd_convex} and \ref{prop:conv_ismd_strconvex} we showed that the expected time average converges to a neighborhood of the optimizer at rate 
$e^{-\mu t}$ for the strongly convex objective (with $\mu$ being the convexity constant for $f$). Proposition \ref{prop:logsobolev_particle} and specifically the result in \eqref{eq:erm} looks directly at $\mathcal{W}(\mathbf{z}_{t})-\mathcal{W}(\mathbf{z^{*}})$ and uses instead $\kappa$, the convexity constant of $\mathcal{V}$, which was also used to bound the fluctuation terms in Section \ref{sec:bound_fluct}. While this means direct comparisons are case specific and depend on the choice of the mirror map $\Phi$, the convergence rates arein both settings exponential. 
In other words, for the strongly convex case both Proposition \ref{prop:conv_ismd_strconvex} and Proposition \ref{prop:logsobolev_particle} lead to exponential rate albeit with different constants due to working with $f$ and $\mathcal{V}$ respectively. 

Furthermore, in terms of the tradeoffs between interaction and convergence, both bounds give a similar result. 
Observe that from Proposition \ref{prop:conv_ismd_strconvex} combined with Proposition \ref{prop:fluctuation_stochastic_strconvex} we obtain, 
\begin{align}
\frac{1}{N}\sum_{i=1}^N\mathbb{E}\left[f(x_\infty^i)-f(x^*)\right] \leq K\left(\frac{\sigma^2}{2N}+\frac{d(N-1)}{(\kappa+\underline{\lambda})N}\sigma^2\right)^{1/2},
\end{align}
where we remark that we have simplified notation to shorten the presentation and we assume $T\gg 0$ and denote with $x_\infty^i$ the corresponding exponential time average. The expectation is furthermore taken with respect to the initial condition. 
Furthermore, from \eqref{eq:erm}, 
\begin{align}
\frac{1}{N}\sum_{i=1}^N\mathbb{E}_{\eta_\infty^N}\left[\mathcal{V}(z^i)-\mathcal{V}(z^*) \right] \leq K\left(\frac{2d}{\kappa}+\frac{\sigma^2}{4N}\log\left(\frac{2L_N}{\sigma^2}\right)\right).
\end{align}
From the above bounds -- while as mentioned different in nature -- we observe that both consist of a term which decreases with $N$, and a term which is of $\mathcal{O}(1)$ in $N$. Both bounds additionally decreases with an increase in $\kappa$, or the strong convexity constant of the objective function.

\subsection{Mapping back to the primal space}

Convergence in the mirror domain and in $z^i_{t}$ should imply convergence
in the primal domain and $x^i_t$. Given $\Phi$ is invertible the law of $X_{t}$
follows from standard change of variable/coordinates of continuous
random variables. This is we have for all $t\geq0$:
\[
\frac{dLaw(x_{t}^{i})}{dx}(x)=\frac{d\eta_{t}^{i,N}}{dx}\circ\nabla\Phi^{*}(x)\left|\det\left(\nabla^{2}\Phi^{*}(x)\right)\right|.
\]
Recall an implication of Proposition \ref{prop:logsobolev_particle} is that $W_{2}(\eta_{t}^{N},\eta_{\infty}^{N})\leq K'e^{-\rho t}$, 
so this is equivalent to 
\[
W_{2}\left(Law(x_{t}^{i}),Law(x_{\infty}^{i})\right)\leq K'e^{-\rho t}.
\]
In the context of Bakry-Emery theory and Markov semigroups we refer
the the interested reader to Section 1.15.1 of \cite{bakry2013analysis}
and to \cite{hsieh18} for a rigorous treatment related to sampling
methods including equivalence in total variation norm. 

\section{Numerical results}\label{sec:num}
In this section we report numerical experiments using two standard benchmark problems.
The first problem is based on linear regression and is an adapted version of the experimental setup that appeared in \cite{beck17book}. The second problem is based on a variation of the classic traffic assignment problem \cite{mertikopoulos18}. We consider compare between (Stochastic) Gradient Descent with Euclidean projections ((S)GD), (Stochastic) Mirror Descent with a Bregman divergence generated by the entropy function. In our numerical implementations of the optimization algorithm we use the Euler discretization of the continuous SMD dynamics:
\begin{align}
z^i_{t+1}= z_t^i-\eta\epsilon\nabla f(x_t^i) + \epsilon\sum_{j=1}^NA_{ij}(z_t^j-z_t^i) + \sigma \sqrt{\epsilon} \mathcal{N}(0,1).
\end{align}
The parameters that we have to choose are the time step $\epsilon$, the learning rate $\eta$, the number of training iterations $T_N$ and the noise level $\sigma$ (set to zero in case of GD or MD). 
Unless otherwise mentioned we set $A_{ij}=\frac{1}{N}$, i.e. a mean-field interaction matrix. 
 
\subsection{Linear system with simplex constraints}\label{sec:num_lin}
We consider a first similar set-up as in Example 9.19 from \cite{beck17book} to compare mirror descent with projected gradient descent. Consider the problem,
\begin{align}
\min_{x\in\mathcal{X}} ||Wx-b||_2^2,
\end{align}
where $\mathcal{X}=\Delta_n$, the unit simplex, $W\in\mathbb{R}^{m\times d}$ and $b\in\mathbb{R}^m$. Unless otherwise mentioned, we set $\epsilon=0.1$ and $T_N=2000$. We generate $A$ randomly with some given condition number $\kappa$, in particular a well-conditioned problem with $\kappa(W)=1$ and an ill-conditioned one with $\kappa(W)=100$.  We let $b_i \sim \mathcal{N}(0,1)$. We set $m=100$ and $d=100$. 

\paragraph{Single particle optimization}
We begin with a comparison of the convergence speed of GD, MD and SGD and SMD, to show that MD/SMD can outperform GD/SGD, as shown in the theoretical results in Lemmata \ref{lem:conv_md_convex}-\ref{lem:conv_md_strconvex} and Propositions \ref{prop:conv_smd_convex}-\ref{prop:conv_smd_strconvex}. We observe from Figure \ref{fig1} that mirror descent algorithms attain a much faster convergence than gradient descent in the system where the matrix $W$ has a high condition number. For a low condition number the performance of mirror descent is similar to that of gradient descent. Mirror descent seems more robust to a higher noise level, and is able to converge even in a high noise setting. We set $\sigma=0.05$. 

\begin{figure}[h!]
\centering
  \includegraphics[width=0.4\textwidth]{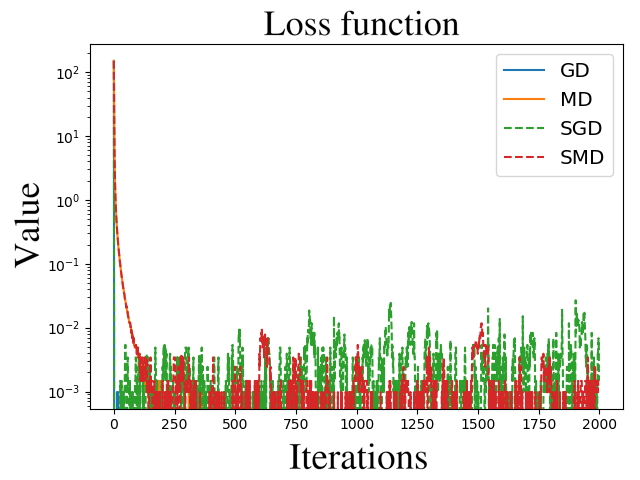}
    \includegraphics[width=0.4\textwidth]{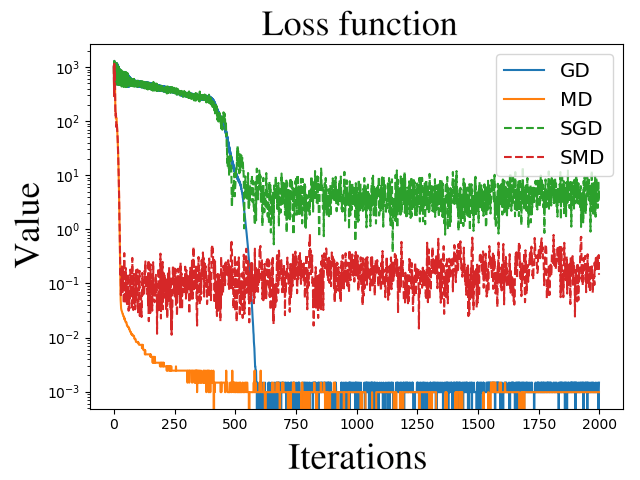}
      \caption{A comparison between GD, MD, SGD and SMD for the linear system with condition number 10 with $\eta = 0.3/\sqrt{t}$ (L) and condition number 100 with $\eta = 0.1/\sqrt{t}$ (R). MD attains a better performance, both in the deterministic as stochastic setting for a high condition number.}\label{fig1}
\end{figure}

\paragraph{Interacting particle optimization}
The theoretical results in Section \ref{sec:gen_bound_ismd} show that the expected value of the distance between the objective evaluated at the time average and the objective evaluated at the optimum is smaller for a larger number of particles. In this section we analyze this numerically and observe that ISMD is indeed able to converge closer to the  minimum. Here we consider the linear system optimized with one, 10 and 100 particles. We set $\sigma=0.05$ and $\eta = 0.01$ for $\kappa(W)=10$ and $\eta = 0.001$ for $\kappa(W)=100$. We remark that we thus use \emph{fixed} learning rates. In Figure \ref{fig2} we show the initial convergence speed. In the setting with a high condition number the convergence speed using interaction can be significantly faster than when considering just a single particle. 

\begin{figure}[h!]
\centering
  \includegraphics[width=0.4\textwidth]{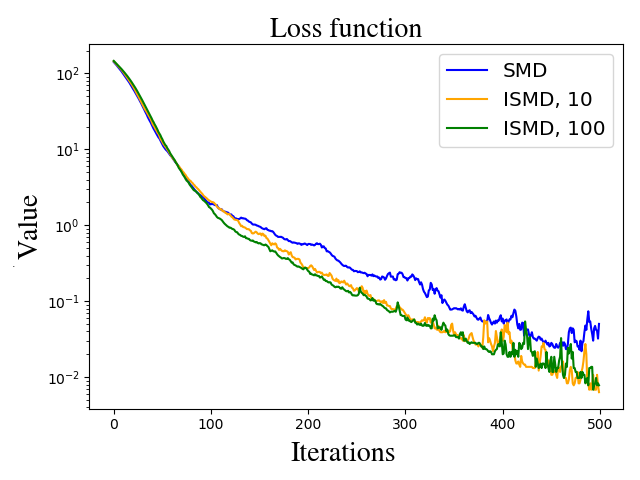}
    \includegraphics[width=0.4\textwidth]{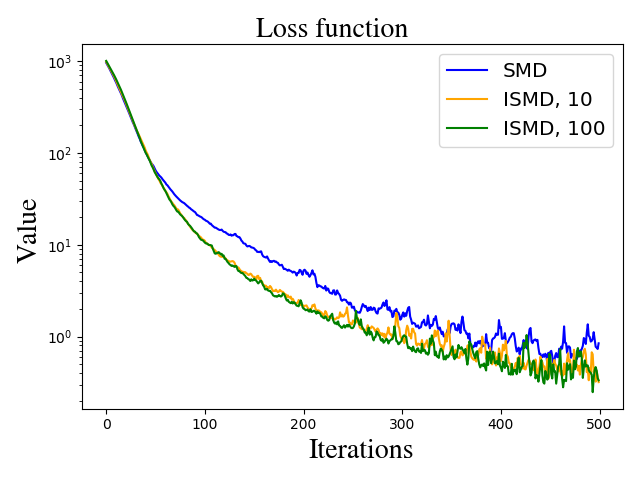}
         \caption{A comparison between the initial convergence of SMD and ISMD for the linear system with condition number 10 (L) and condition number 100 (R). We observe a speedup in convergence using interacting particles.}\label{fig2}
\end{figure}

In Figure \ref{fig3} we show the convergence speed as well as the distribution of the loss $f(x_t^i)-f(x^*)$ for the linear system with a condition number 100. 
We observe from the plots and histograms that the noise in the interacting sampler is smaller so that it able to converge to closer to the global minimum than single-particle optimization. More specifically, the more particles used, then the smaller the variance of the samples is, so that the values are closer to the true minimizer. 


\begin{figure}[h!]
\centering
    \includegraphics[width=0.4\textwidth]{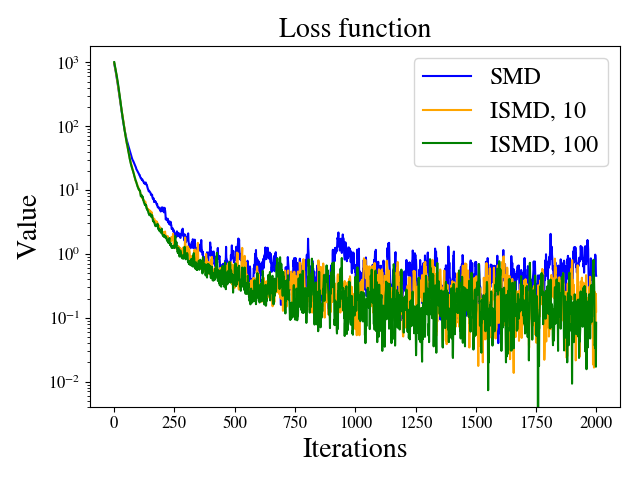}
         \includegraphics[width=0.4\textwidth]{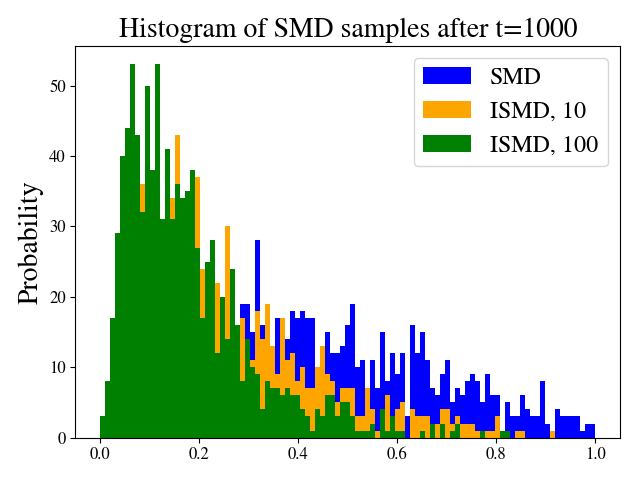}
         \caption{A comparison between SMD and ISMD for the linear system with condition number 100: (L) the loss function, (R) the histogram which shows the distribution of the losses $f(x_t^i)-f(x^*)$ after convergence ($t>1000$).  
         }\label{fig3}
\end{figure}

\paragraph{Convergence: vanishing noise against interaction}
In this example we study how using a vanishing noise compares to using a fixed learning rate with and without interaction. Often in practice, a vanishing noise is achieved by using a larger batch size, so that the gradient noise is smaller. In Figure \ref{fig4} the results for the distance to the optimum, $f(x_t^i)-f(x^*)$, are shown. As usual, the convergence speed using interaction is higher and using interaction can achieve the same effect as using a vanishing noise in terms of the distance to the optimum. We remark here that using an interaction strength, i.e. letting the interaction term be given by $\theta\sum_{i=1}^NA_{ij}(z_t^j-z_t^i)$ can further help in controlling the noise variance. The effects and necessity of such an interaction term will be studied in future work. An additional benefit of having noise is that noise allows to escape local minima in a non-convex problem. Being able to control this noise without decreasing the convergence speed is a significant benefit. 

\begin{figure}[h!]
\centering
  \includegraphics[width=0.4\textwidth]{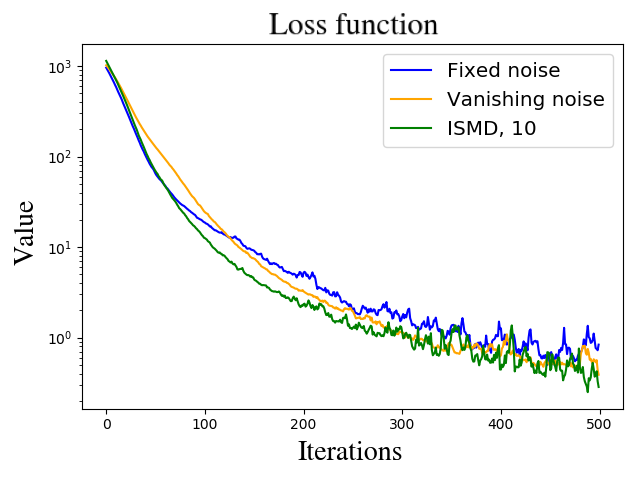}
    \includegraphics[width=0.4\textwidth]{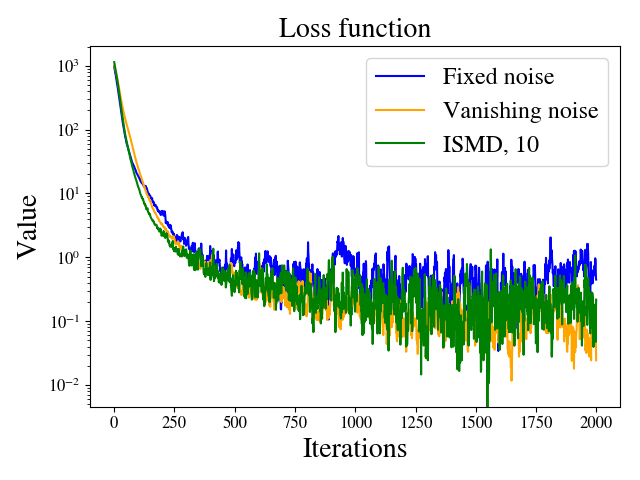}
         \caption{The initial distance to the optimum (L) and the distance to the optimum after convergence (R) for the linear problem with condition number 100 for SMD with a fixed noise $\sigma=0.05$, SMD with a vanishing noise $\sigma = \sigma/(t+1)^{1/10}$ and ISMD with 10 particles.}\label{fig4}
\end{figure}

\paragraph{Variance reduction}
As implied by the theory in Section \ref{sec:gen_bound_ismd} and the numerical results in the last section, the expected value of the distance between the objective evaluated at the time average and the objective evaluated at the optimum is smaller for a larger number of particles. 
We keep the learning rate fixed at $\lambda=0.001$ and set the condition number to $100$. To make a fair comparison we sample $N$ i.i.d. copies of the SMD algorithm, and use $N$ particles in the ISMD algorithm. We first plot the variance of $f(x_t^i)-f(x^*)$ after convergence ($t>1000$) for the i.i.d. an interacting setting in the left-hand side of Figure \ref{fig5}. Clearly, in the non-interacting setting the variance is not decreased when increasing the number of particles and due to the algorithm not having converged properly the variance is fluctuating. Even more so, the evolution of the variance of the loss is chaotic since the number of particles does not influence it. In the interacting setting more particles results in a lower loss variance. 
To understand the effects of interaction and the number of particles on the fluctuation term we compute $||x_T^i-\bar x_t^N||_2$, whose results are shown in the right-hand side of Figure \ref{fig5}. As expected from the theoretical analysis the deviation from the mean is smaller in the interacting case than in the i.i.d. case. Furthermore, as $N$ increases the deviation remains constant, validating the theoretical results in Section \ref{sec:bound_fluct} that the fluctuation is bounded and a non-increasing function of $N$. 

\begin{figure}[h!]
\centering
  \includegraphics[width=0.4\textwidth]{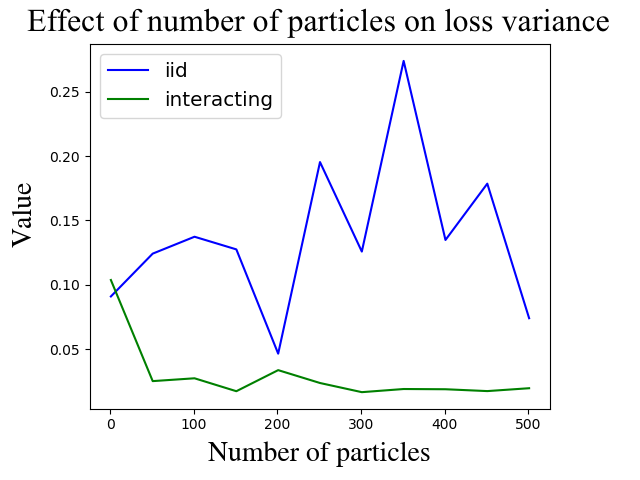}
  \includegraphics[width=0.4\textwidth]{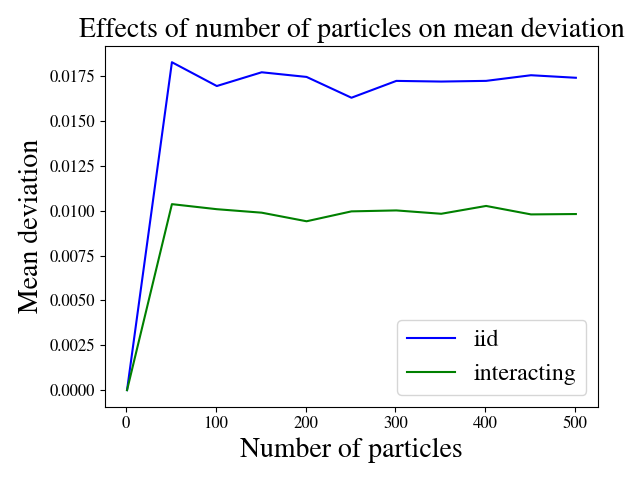}
         \caption{The variance of the loss trajectory as a function of the amount of particles (L) and the mean of the $L_2$-norm of the fluctuation term $||x_T^i-\bar x_t^N||_2$ for the i.i.d. and interacting case. We see that in the interacting case increasing the number of particles decreases the variance. The deviation from the mean is smaller in the interacting case, as expected. }\label{fig5}
\end{figure}


\subsection{Traffic Assignment Problem}
The objective of the traffic assignment problem is to compute the optimal path between two nodes in a graph. To save space we refer the reader to \cite{mertikopoulos18} for a precise description of the problem. We only mention that the problem is a convex optimization problem with a simplex constraint and therefore fits our framework. 

Our experiments below are based on a random geometric graph $G(n,r)$ i.e. we placed $n$ points uniformly at random in $[0, 1]^2$ and connect any two points whose distance is at most some $w$. The weights of the edges in the graph also act as the weight of each edge. We randomly chose two nodes to act as the origin and destination nodes. To set up the optimization problem we also need to compute all the simple paths between the origin and destination node. We denote the length of the maximum path considered in this phase as $r_{\max}$. The size of $r_{\max}$ decides the dimensionality of the problem. For the stochastic version of the problem we added Gaussian $N(0,\sigma_g)$ noise on each edge.
We used a constant step-size for all algorithms. For mirror descent (all variations) we used $\eta=0.2$ and for gradient descent we manually selected the best step-size for the problem (see below for additional remarks regarding GD for this problem).

\paragraph{Comparison between MD and GD}
An interesting difference between the regression problem above and the traffic assignment problem is that the solution of the latter is sparse. As a result the solution is on the boundary with many paths having zero load. 
In this case the gradient descent algorithm performed considerably worse than all variations of mirror descent (i.e. with many or single particles) and in all variations of the problem we considered (small/large problems, with and without noise).  
In order to illustrate the difference in performance between mirror descent and gradient descent we plot the results for one example in Figure \ref{fig:deterministic_md_gd}. The graph used to generate the problem above has $100$ nodes, and the resulting optimization problem has $5831$ possible paths i.e. $d=5831$.
We begin with an illustration of the performance of standard MD against GD in Figure \ref{fig:deterministic_md_gd}; GD is significantly slower to converge than MD. Given the large difference in performance between MD and GD we will not report any more results from GD for the traffic assignment problem.

\begin{figure}[h!]
\centering
\includegraphics[width=0.4\textwidth]{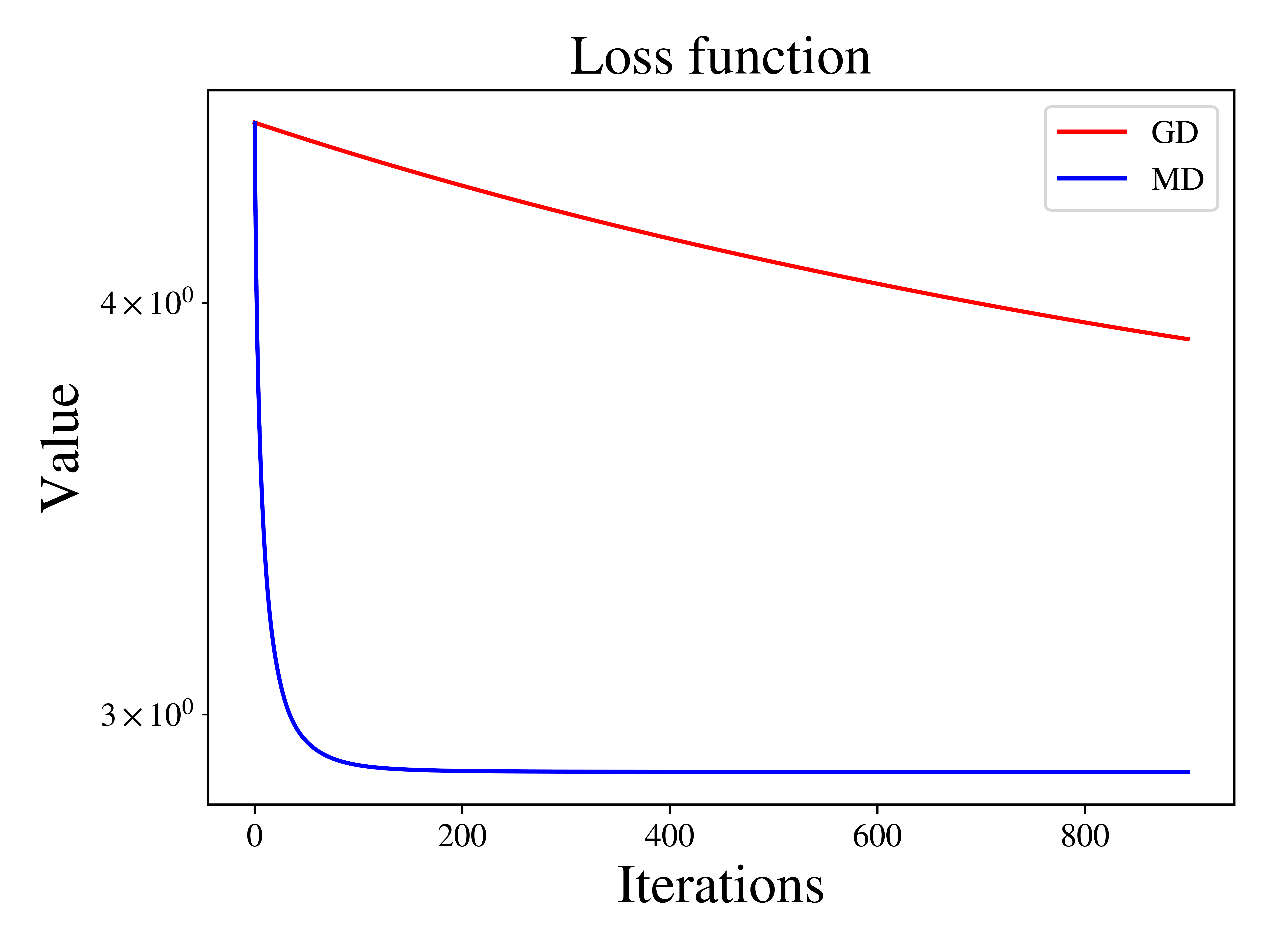}
\caption{A comparison between MD and GD. GD with Euclidean projections is considerably slower than MD for this problem because the solution is sparse. IMD performs the same as MD for deterministic problems.}\label{fig:deterministic_md_gd}
\end{figure}

\paragraph{Comparison between SMD and ISMD}
Similar to the experiments run for the linear regression we compare the iterations of the single particle mirror descent i.e. (SMD) and its interacting variant (ISMD) with a different number of particles $(N=10,100,1000)$. 
In the left of Figure \ref{fig:path_smd_ismd} we show the iteration history of (SMD) and (ISMD) with $N=100$ particles. To visualize the differences between (SMD) and (ISMD) and the impact of the number of particles we plot a histogram in the right plot in Figure \ref{fig:path_smd_ismd}  of the iterations after convergence (typically $t>100$). Clearly there is a considerable reduction in the variance of the iterates between (MD) and (ISMD). We attempt to quantify the reduction in variance in the next set of experiments.

\begin{figure}[h!]
\centering
\includegraphics[width=0.42\textwidth]{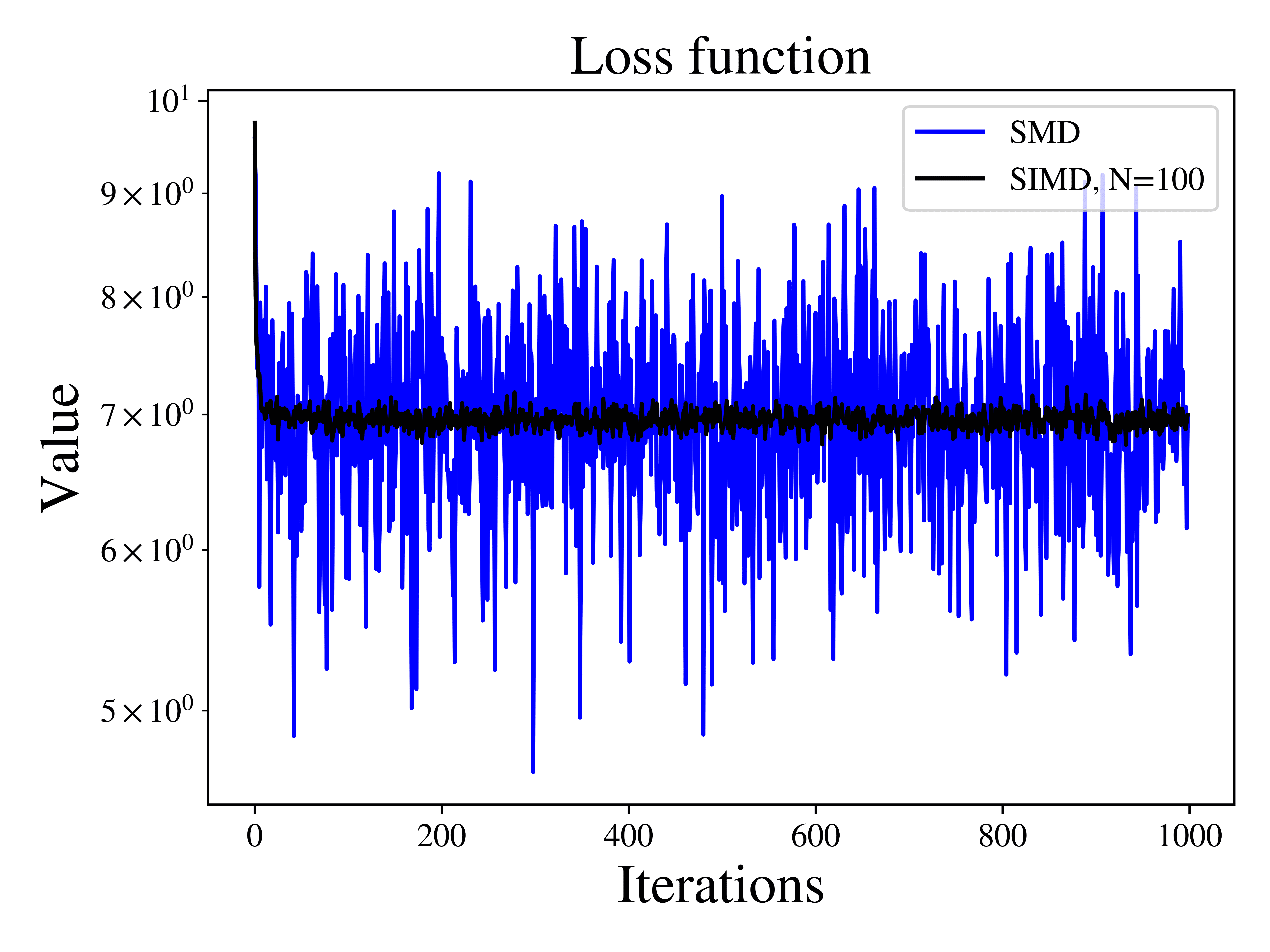}
\includegraphics[width=0.4\textwidth]{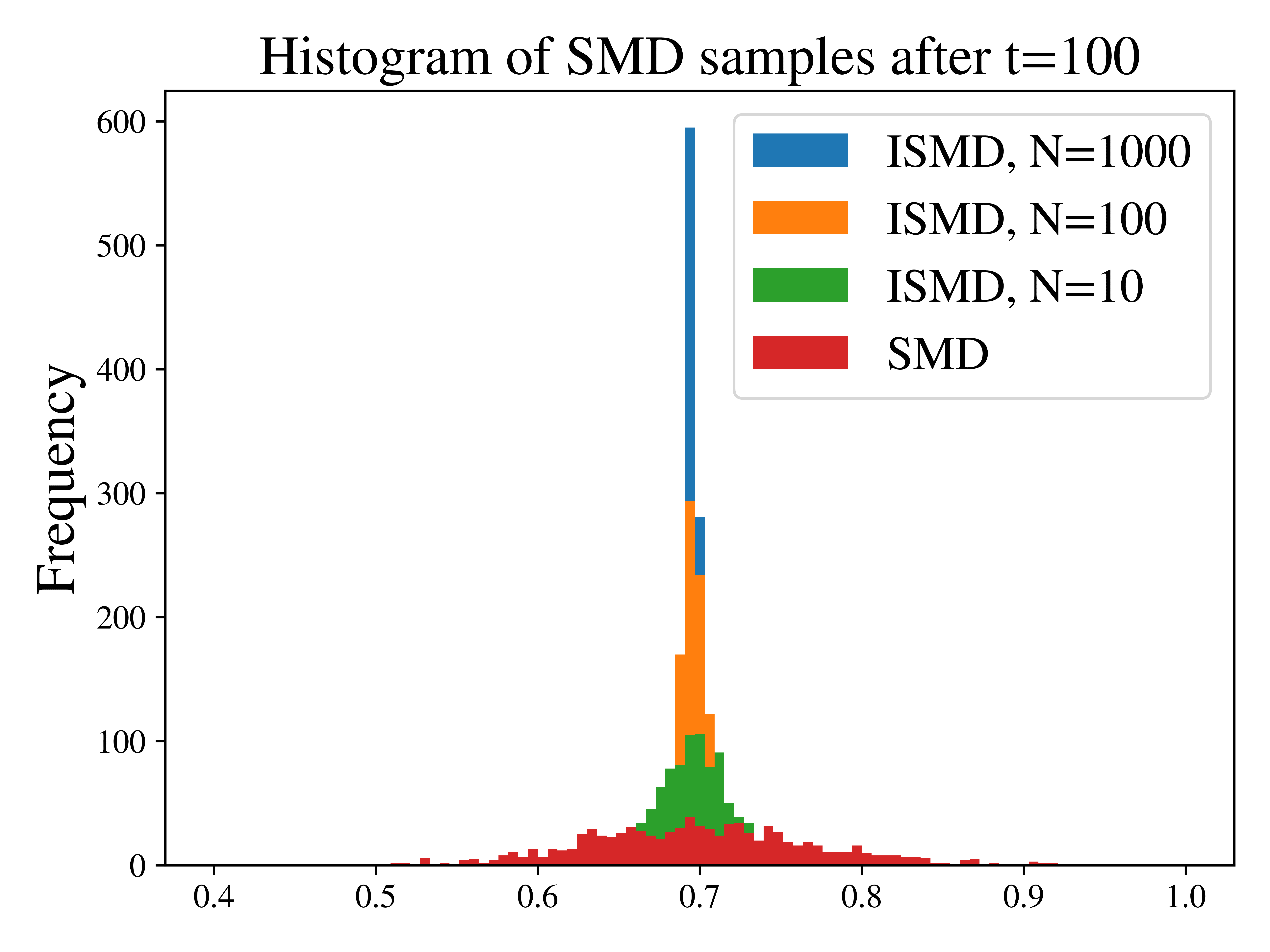}
\caption{Comparison between SMD and ISMD for a single path (L) and a histogram of the samples for the comparison between SMD and ISMD with 10, 100 and 1000 particles (R). With more particles the variance of the samples is lower.}\label{fig:path_smd_ismd}
\end{figure}

In order to further quantify the variance reduction between ISMD and SMD in a fair way it is important to compare the reduction in variance between $N$ i.i.d. copies of SMD and ISMD with $N$ particles. In Table \ref{tab1} we present the variance reduction results for different particles and observe that interaction can significantly reduce the variance. 

\begin{table}[H]
\begin{center}
\begin{tabular}{lllll}
 Nodes & $r_{\max}$ & $d$ & $N$ & $\tfrac{\sigma^2_{SMD}-\sigma^2_{ISMD}}{\sigma^2_{SMD}}$\\ \hline\hline
50 & 5 & 70 & 10 & 0.18 \\
50 & 5 & 70 & 50 & 0.21 \\
50 & 5 & 70 & 100 & 0.09\\
\end{tabular}
\end{center}
\caption{A comparison of the variance of $N$ i.i.d. copies of SMD and $N$ particles in ISMD. We obtain a variance reduction using interaction.}\label{tab1}
\end{table}

\subsection{Mini-batch optimization}\label{sec:mini_batch}
In machine learning the optimization objective typically consists of a sum over data samples, i.e. $f(x) = \sum_{i=1}^m f_i(x)$, where $m$ is the sample size. The gradient is typically computed over a subset of the data (a mini-batch), since for large $m$ computing the full gradient is too costly. The downside of computing the gradient over a mini-batch is that the gradient contains noise. In our work we proposed interaction as a way of reducing the effects of this gradient noise. In this section we show how interaction helps in reducing the effect of noise in mini-batch gradient descent, and consider in more detail the computational costs related to interaction against an increased batch size. The ideas presented here can be generalized to popular machine learning techniques such as neural networks.  

Consider a similar setup as in Section \ref{sec:num_lin},
\begin{align}
f(x) = \frac{1}{m}\sum_{i=1}^m f_i(x):=\frac{1}{m}\sum_{i=1}^m ||W_{i,\cdot}x-b_i||_2^2.
\end{align}
In every iteration, the gradient is computed over a subset of the data, $f_\mathcal{S}(x)=\frac{1}{|\mathcal{S}|}\sum_{i\in\mathcal{S}}f_i(X)$, where $|\mathcal{S}|$ refers to the size of the batch. The optimization algorithm is then given by, 
\begin{align}
z^i_{t+1}= z_t^i-\eta\epsilon\nabla f_\mathcal{S}(x_t^i) + \epsilon\sum_{j=1}^NA_{ij}(z_t^j-z_t^i).
\end{align}
The noise is thus implicit in the gradient $\nabla f_\mathcal{S}$. 

\paragraph{Convergence speed} 
In Figure \ref{fig7} the convergence results are presented for different batch sizes $|\mathcal{S}|$ with and without interaction. We set $m=200$ and $d=100$, and use $\kappa(W)=10$ with $\eta=0.1$ and $\kappa(W)=200$ with $\eta=0.1$. Using $N$ particles infers a computational cost of $N$ times that of a single particle. However, as observed in the plot, using $N$ particles allows to use a significantly smaller batch, i.e. due to the interaction the optimization can converge even with a higher noise. We conclude that interaction can be a convenient methodology for improving the convergence of the algorithm when using stochastic gradients. 

\begin{figure}[h!]
\centering
\includegraphics[width=0.4\textwidth]{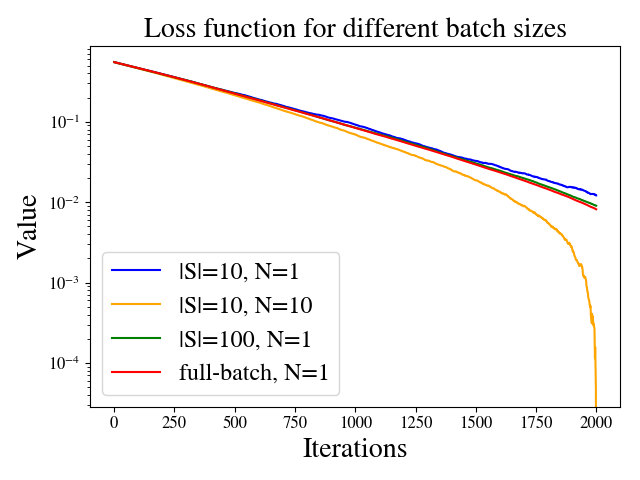}
\includegraphics[width=0.4\textwidth]{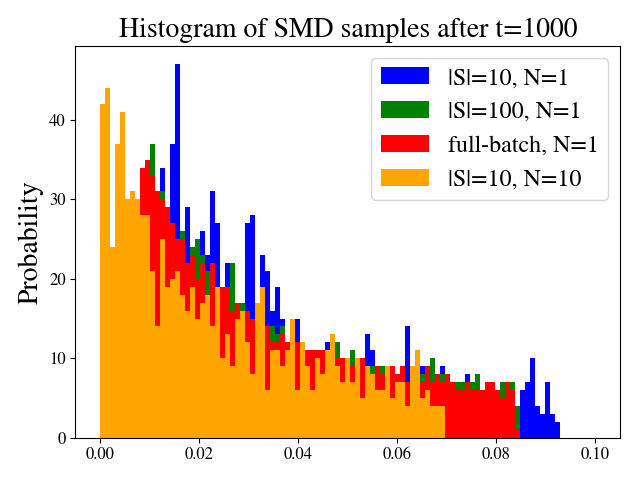}\\
\includegraphics[width=0.4\textwidth]{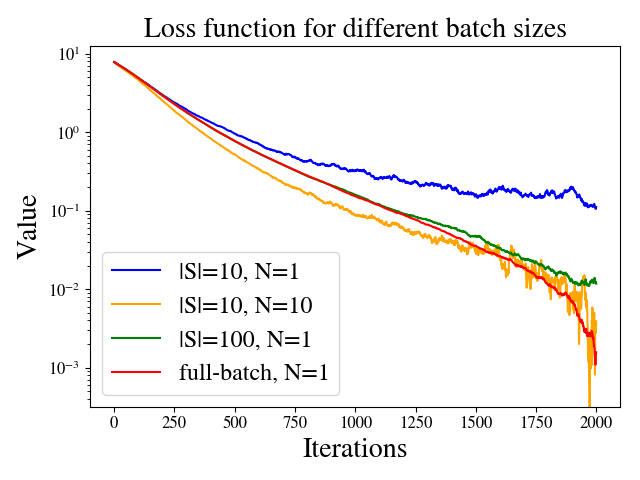}
\includegraphics[width=0.4\textwidth]{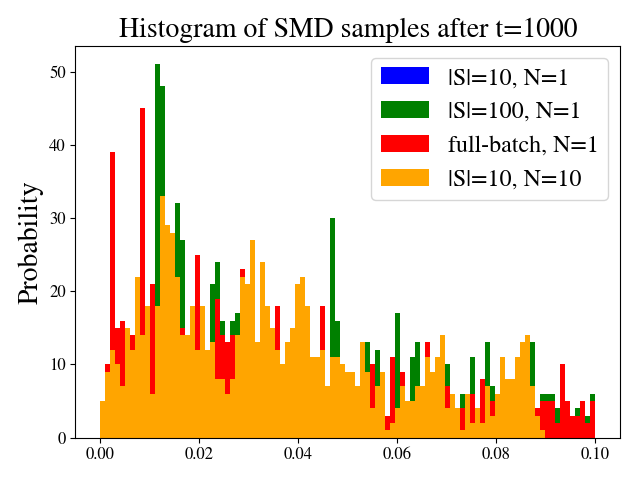}
\caption{The loss function (L) and the histogram (R) for ISMD with different batch sizes and different particles for $\kappa(W)=10$ (top) and $\kappa(W)=200$ (bottom). The loss function is computed as a sum over samples, and the gradient is computed over a batch of samples. Using interacting particles allows to use a smaller batch size while still attaining convergence. The presented results are averaged over 10 runs.}\label{fig7}
\end{figure}

\paragraph{Changing the interaction matrix}
In this section we study the trade-offs between interaction matrix sparsity and convergence time. We use a batch size of 10 and 10 nodes, and $\kappa(W) = 200$ with $\eta=0.01$. Table \ref{tab2} shows the results for a Erdős-Rényi communication graph, a random graph where each edge is chosen with a certain probability. A connectivity probability of $p$ means that on average each node is connected to $p\times 10$ other nodes. The communication between the nodes is thus determined by the connectivity probability. The weight matrix $A$ is determined by a doubly stochastic version of this communication graph. From the results we see, as expected, if more communication is present the convergence is faster at the cost of more communication. Interestingly, the total communication is lowest for the most sparse graph, showing that fast convergence can even be achieved with minimal communication. We do remark that this depends on the particular setup of the problem: in a same $f$ setting the convergence can be equally fast for sparse and dense communication (not shown here) while for highly different local functions $f_i$ more communication may be needed to achieve convergence. In Figure \ref{fig8} we plot the convergence speed for the different interaction matrices. In correspondance with the results in Table \ref{tab2} we observe a faster initial convergence for the densely connected matrix. Interestingly, even with sparse communications a relatively fast convergence can be obtained, showing the potential of our methods.
\begin{table}[h!]
\begin{center}
\begin{tabular}{c||c|c|c}
$\mathbb{P}(\textnormal{connectivity})$ & 0.3& 0.5 & 1\\\hline\hline
Time to convergence & 883 & 712 & 347\\
Total communication for convergence & 2649 & 3560 & 3470\\
Loss at $t=2000$& 42.93 & 42.83 & 41.93\\
Consensus error at $t=2000$ & 0.013 & 0.012 & 0.011\\
\end{tabular}
\end{center}
\caption{A comparison of the communication costs and convergence time for different interaction graphs. The results are averaged over 10 runs of the algorithm (both the generation of the graph as the optimization). Time to convergence is determined as the first time the loss value hits below a certain level (here 43.02). The total communication for convergence is determined as the number of iterations needed to converge times the average communication per round.}\label{tab2}
\end{table}

\begin{figure}[h!]
\centering
\includegraphics[width=0.4\textwidth]{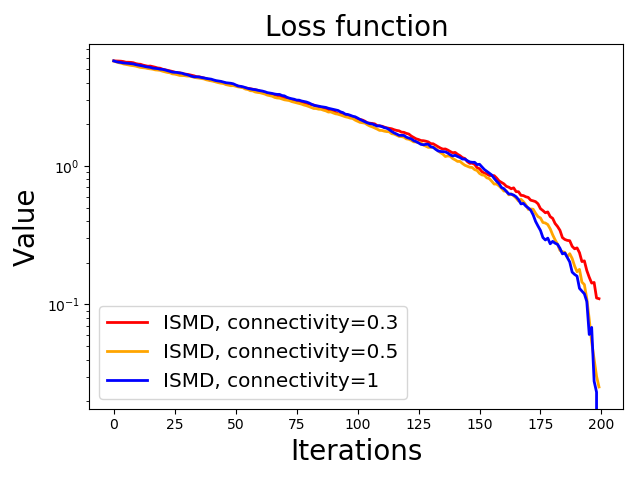}
\includegraphics[width=0.4\textwidth]{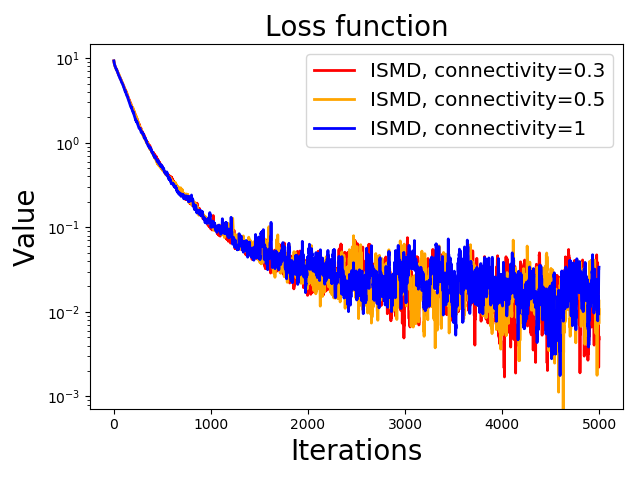}
\caption{The loss function for initial convergence (L) and full convergence (R) for different interaction matrices. Interstingly, even with sparse communication a fast convergence can be achieved.}\label{fig8}
\end{figure}

\section{Conclusion}\label{sec:concl}

In this work we analyzed stochastic mirror descent with interacting particles with the aim of understanding the tradeoffs between computation, communication and variance reduction. With a fixed learning rate or a constant noise variance convergence to the exact minimizer is not possible. Prior results showed that decreasing the learning rate \cite{shi2020learning} or using a vanishing noise variance \cite{mertikopoulos18} can result in good convergence properties of stochastic optimization. In our work we showed that using $N$ interacting particles can in a similar fashion decrease the effect of noise and result in a closer convergence to the optimum. In addition using interaction decreases the variance of the samples further than that of $N$ independent replicas of SMD. We therefore argue that interaction between $N$ particles is a viable method and can control the convergence of the algorithm. 

We analyzed the convergence from an optimization perspective exploiting Lyapunov functions. We showed that the deviation of an individual particle from the particle mean is bounded and showed that ISMD achieves a linear convergence rate in the convex case and an exponential one in the strongly convex setting. We furthermore presented the relationship between sampling from the invariant measure and converging to the optimizer of the objective function. Using log-Sobolev inequalities we obtained an explicit convergence rate to the invariant measure, implying that the particles converge to a neighborhood of the optimizer at an exponential rate. Our numerical results show that interaction is beneficial in cases where the solution is sparse and therefore the behavior of the algorithm at the boundary is important. This is a case where the role of the mirror map becomes more critical, which is also reflected in our assumptions. 

Interesting extensions include the study of SMD in the non-convex setting where the noise can help escape from local minima (see e.g. \cite{zhu18}, \cite{helffer1998remarks}). In this case the interaction can play the role of a regularizer due to the convexification of the loss surface \cite{zhang15}. Additionally,even in a non-convex setting using interaction enforces the particles to converge to consensus in a local minimum. The interaction can also act as a pre-conditioner and its statistical properties may help speed up the convergence or minimize the variance . An open question is the study of the optimal interaction matrix beyond just mixing times \cite{boyd2004fastest}. Alternatively of interest is to study the choice of the interaction such that communication is minimal, but variance reduction and consensus is still guaranteed.

It is also interesting to study such algorithms in a truly distributed setting, i.e. when each particle has access to a subset of the data. This is the case in learning with privacy constraints or federated learning approaches. Some preliminary results on this have been presented in \cite{borovykh20icml}. In this setting also the effect of an interaction strength will be of interest and we aim to address the necessity of considering an additional interaction strength $\theta$, i.e. using $\theta\sum_{i=1}^NA_{ij}(z_t^j-z_t^i)$, in future work. More specifically, in a distributed setting the particles optimize different objectives, and -- unlike in our setting of a same objective -- convergence to consensus cannot be guaranteed without imposing a sufficiently high interaction strength.


Lastly, the effect of time discretization was not considered here. This is a well-understood topic. We refer the interested reader to \cite{malrieu2006concentration}, \cite{veretennikov2006ergodic} for details related to the Euler discretization and \cite{gu18} for more advanced schemes which also include acceleration.

\paragraph{Acknowledgements}
This project was funded by JPMorgan Chase $\&$ Co under a J.P. Morgan A.I. Research Award 2019. Any views or
opinions expressed herein are solely those of the authors listed, and may differ from
the views and opinions expressed by JPMorgan Chase $\&$ Co. or its affiliates. This
material is not a product of the Research Department of J.P. Morgan Securities
LLC. This material does not constitute a solicitation or offer in any jurisdiction.
G.A.P. was partially supported by the EPSRC through the grant number EP/P031587/1.

\appendix
\section*{Appendix}

\section{Auxiliary results for Section \ref{sec:ismd}} \label{sec:lemma_conv_md}
Here we proof the Lemma's that were used in Section \ref{sec:ismd}.
In Proposition \ref{prop:conv_ismd_convex} we used the following lemma. 
\begin{lemma}\label{lem:conv_md_convex0}
Under the continuous stochastic mirror descent dynamics \eqref{eq:smd} it holds,
\begin{align}
\int_0^T \frac{1}{N} \sum_{i=1}^N(y_t^N-x^*)^T\nabla f(x_t^i) dt \leq \frac{1}{2}D^2_{\Phi,\mathcal{X}} + \frac{\sigma^2T}{2N}\Delta \Phi^*(\bar z_t^N)+\int_0^T \frac{\sigma}{\sqrt{N}}||y_t^N-x^*||_2dB_t^i.
\end{align}
\end{lemma}
\begin{proof}
Consider the Lyapunov function defined as 
\begin{align}
V(\bar z_t^N)=\Phi^*(\bar z_t^N)-\Phi^*(z^*)-(x^*)^T(\bar z_t^N-z^*).
\end{align}
Applying It\^o's lemma and using the fact that $\sigma (y_t^N-x^*)^T\frac{1}{N}\sum_{i=1}^NdB_t^i\overset{d}{=}\frac{\sigma}{\sqrt{N}}||y_t^N-x^*||_2dB_t$ we obtain,
\begin{align}
dV(\bar z_t^N)=& d\Phi^*(\bar z_t^N)-(x^*)^Td\bar z_t^N\\
=&(\nabla \Phi^*(z^*)-\nabla\Phi^*(\bar z_t^N))^T\frac{1}{N}\sum_{i=1}^N\nabla f(x_t^i)dt + \frac{\sigma^2}{2N}\Delta \Phi^*(\bar z_t^N)dt+\frac{\sigma}{\sqrt{N}}||y_t^N-x^*||_2dB_t.
\end{align}
Integrating, and using the standard bounds we obtain,
\begin{align}
\int_0^T \frac{1}{N} (y_t^N-x^*)^T\sum_{i=1}^N\nabla f(x_t^i) dt \leq \frac{1}{2}D^2_{\Phi,\mathcal{X}} + \frac{\sigma^2T}{2N}\Delta \Phi^*(\bar z_t^N)+\int_0^T \frac{\sigma}{\sqrt{N}}||y_t^N-x^*||_2dB_t.
\end{align}
\end{proof}

Similar to the convex case, in Proposition \ref{prop:conv_ismd_strconvex} we used the following lemma. 
\begin{lemma}\label{lem:conv_md_strconvex0}
Under the continuous stochastic mirror descent dynamics \eqref{eq:smd} it holds,
\begin{align}
\int_0^T e^{\mu t} \left(\frac{1}{N} (y_t^N-x^*)^T\sum_{i=1}^N\nabla f(x_t^i)  - \mu D_{\Phi^*}(\bar z_t^N,z^*)\right) dt &\leq \frac{1}{2}D^2_{\Phi,\mathcal{X}} + \frac{\sigma^2}{2N}(e^{\mu T}-1) \Delta \Phi^*(\bar z_t^N)\\
&+ \int_0^T e^{\mu t}\frac{\sigma}{\sqrt{N}}||y_t^N-x^*||_2dW_t.
\end{align}
\end{lemma}
\begin{proof}
Consider the Lyapunov function defined as $\bar V(t,\bar z_t^N)=e^{\mu t} D_{\Phi^*}(\bar z_t^N,z^*)$. Applying It\^o's lemma we obtain,
\begin{align}
dV(t,\bar z_t^N)=&(e^{\mu t}(x^*-y_t^N)^T\frac{1}{N}\sum_{i=1}^N\nabla f(x_t^i)dt + \mu e^{\mu t}D_{\Phi^*}(\bar z_t^N,z^*)dt + \frac{\sigma^2}{2N}e^{\mu t}\Delta \Phi^*(\bar z_t^N)dt \\
&+\frac{\sigma}{\sqrt{N}}e^{\mu t}||y_t^N-x^*||_2dW_t,
\end{align}
where we have used the evolution of $\bar z_t^N$. 
Integrating and using the standard bounds we obtain the statement. 
\end{proof}

\section{Auxiliary results for Section \ref{sec:sample}}

\subsection{A Primer on Log Sobolev inequalities and Markov semigroups}\label{sec:log_sob_primer}


Recall, we defined in Section \ref{notations} that a probability measure $\nu$ satisfies a Log-Sobolev inequality with constant $C$ is for any smooth function $f$ we have,
\begin{align}\label{eq:log_sobolev_def2}
\textnormal{Ent}_\nu(f^2)\leq C\nu(|\nabla f|^2),
\end{align}
where the entropy was defined as,
\begin{align}
\textnormal{Ent}_\nu(f^2) = \nu(f^2\log f^2) - \nu(f^2)\log(\nu(f^2)).
\end{align}

Satisfying \eqref{eq:log_sobolev_def2} has many important consequences
summarized below.

\begin{theorem}\label{thm:lsi}
Let $\nu\in\mathcal{P}(\mathbb{R}^d)$ satistfy a log Sobolev inequality with constant $C$. Then we have the following:
\begin{enumerate}
\item Exponential integrability of Lipschitz functions (Herbst argument):
Let $f$ be a 1-Lipschitz function with $\left\Vert f\right\Vert _{Lip}\leq1$,
then for every $s\in\mathbb{R}$
\begin{align}
\nu\left(e^{sf}\right)\leq e^{s\nu(f)+Cs^{2}}
\end{align}
and $\nu\left(e^{r^{2}f}\right)<\infty$ for $r^{2}<C^{-1}$; 
\item A concentration of measure result: 
\begin{equation}
\nu\left(\left|f-\mathbb{E}_{\nu}\left[f\right]\right|\geq r\right)\leq2\exp\left(-\frac{r^{2}}{C\left\Vert f\right\Vert _{Lip}^{2}}\right)\label{eq:concentration}
\end{equation}
 for any Lipschitz $f$; .
\item The Wasserstein distance is bounded by relative entropy: $W_{2}(\nu,\tilde\nu)^{2}\leq\frac{C}{2}H\left(\nu|\tilde\nu\right)$
for $\ensuremath{\nu\ll\tilde\nu}$; 
\end{enumerate}
\end{theorem}
\begin{proof}
We provide references for each point: 1. see Proposition 5.4.1 \cite{bakry2013analysis}; 2. see \cite{ledoux1999concentration}; 3. see \cite{otto2000generalization}.
\end{proof}

Note point 1. in Theorem \ref{thm:lsi} implies finite moments for Lipschitz $f$ and $p\geq2$:
\begin{align}\label{eq:sobolev_finite_moment}
\nu\left(\left|f\right|^{p}\right)^{2}\leq\nu\left(\left|f\right|^{2}\right)^{2}+2C(p-2)\left\Vert f\right\Vert _{Lip}^{2},
\end{align}
see Proposition 5.4.2 in \cite{bakry2013analysis} for details.

Proving Theorem \ref{thm:lsi} uses properties of certain Markov semigroups and the corresponding Fokker Planck equations. The converse is also possible. One can establish the log-Sobolev inequalities for the law of the process one can employ the classical results from \cite{bakry1997sobolev}. We briefly summarize some useful results for an SDE of the form,
\begin{align}\label{eq:z_dynamics_appx}
dz_{t}=-\nabla\mathcal{V}(z_{t})dt+\sqrt{2}dB_{t},\quad z_{0}=z
\end{align}
Let $\nu_t=Law(z_t)$. In order to demonstrate ergodicity and exponential convergence to equilibrium, one approach is to require existence of a ($L^2$) spectral gap. A common requirement $\lim_{\left|z\right|\rightarrow\infty}\left(\frac{\left|\nabla\mathcal{V}\right|^{2}}{2}-\Delta\mathcal{V}\right)=\infty$ and this ensures that $\nu_t$ satisfies a Poincare inequality and the SDE is ergodic with invariant density 
\begin{align}
\nu_\infty (dz) := \frac{1}{Z}\exp(\mathcal{V}(z))dz,
\end{align}
see Proposition 4.2 and Theorem 4.3 in \cite{pavliotis14book} for more details. 
One can strengthen these results with log-Sobolev inequalities from \cite{bakry1997sobolev}:

\begin{theorem}\label{thm:bakry-emery}
Assume 
\begin{align}\label{eq:V_convex}
\textnormal{Hess}\mathcal{V}\succeq\rho I_d,\quad \rho\in\mathbb{R}
\end{align} 
Then $\nu_t$ satisfies \eqref{eq:log_sobolev_def}
with constant 
\begin{align}
C_{t}=\frac{2}{\rho}\left(1-e^{-\rho t}\right).
\end{align}
If in addition $\rho>0$, we get,
\begin{align}
H(\nu_t|\nu_\infty )\leq Ke^{-2\rho t}.
\end{align}
\end{theorem}

Note that requiring $\mathcal{V}$ being strongly convex is crucial to obtain explicit rates of convergence to equilibrium. This is a particular instance of the celebrated (and more general) Bakry-Emery criterion on the curvature of Markov semigroups. To state the more general version for any Markov semigroup we need to define the following differential operators defined on appropriate domains:
\[
\mathfrak{L}(f)=-\nabla\mathcal{V}^T\nabla f +\frac{\sigma^2}{2} \textnormal{Hess} (f), \Gamma(f)=\frac{1}{2}\left(\mathfrak{L}(fg)- f\mathfrak{L}(g)- g\mathfrak{L}(f)\right), \Gamma_2(f)=\frac{1}{2}\left(\mathfrak{L}(\Gamma(f,f))-2\Gamma((\mathfrak{L}(f),f) \right).
\]
The Bakry-Emery criterion consists of verifying,
\[\Gamma_2(f)\succeq \rho\Gamma(f,f),\] 
for any smooth $f$, which is equivalent to  \eqref{eq:V_convex} for the dynamics in \eqref{eq:z_dynamics_appx}.

In addition, point 3. in Theorem \ref{thm:lsi} also implies a convergence in $W_{2}$, i.e. 
\begin{align}
W_2(\nu_t,\nu_\infty)\leq \sqrt{\frac{K}{2}}e^{-\rho t}.
\end{align}
and by Pinsker's inequality we also have convergence in total variation with same rate:
\begin{align}
||\nu_t- \nu_\infty||_{TV} \leq  \sqrt{\frac{K}{2}} e^{-\rho t},
\end{align}
where we denote $||\nu||_{TV}=\sup_{|\varphi|<1}|\nu(\varphi)|$, for any $\nu$-measurable function $\varphi$.


\subsection{Empirical risk bounds for Proposition \ref{prop:logsobolev_particle}} \label{appx:erm}

The main purpose of this section is to provide so called \emph{empirical risk bounds} for $\mathbb{E}_{\eta_{\infty}^{N}}(\mathcal{W})- \mathcal{W}(\mathbf{z^*})$. We will base our derivation on Proposition 3.4 in \cite{raginsky2017non} and specifically focus on certain aspects of the particle system with invariant distribution 
\[
\eta^{N}_\infty(d\mathbf{z})=\frac{1}{Z^{N}}\exp\left(-\frac{2}{\sigma^{2}}\mathcal{W}(\mathbf{z})\right)d\mathbf{z}.
\]
We will define the differential entropy for
a probability measure $\eta\in\mathcal{P}_{2}(\mathbb{R}^{d})$ as
$\mathcal{H}(\eta)=-\int\log\left(\frac{d\eta}{dz}\right)(z)\eta(dz)$.
From the definition of $\mathcal{H}$ we have for $\eta^{N}_\infty$
\[
\mathbb{E}_{\eta^{N}_\infty}\mathcal{W}=\frac{\sigma^{2}}{2}\left(\mathcal{H}(\eta^{N}_\infty)-\log Z^{N}\right).
\]
In what follows we will provide upper bounds for $\mathcal{H}(\eta^{N}_\infty)$ and lower bounds for $\log Z^{N}$ to reach an empirical risk bound.
\begin{lemma}\label{lem:entropy}
Let $\eta\in\mathcal{P}_{2}(\mathbb{R}^{d})$ and let $\Sigma$ denote the covariance matrix of $\eta$. Then $\mathcal{H}(\eta)\leq\frac{d}{2}\left(\log2\pi+1\right)+\frac{1}{2}\log\det\Sigma$.
\end{lemma}
\begin{proof}
Consider the Kullback Leibler divergence between $\eta$ and a Gaussian distribution with the same mean and variance $\mu,\Sigma$:
\[
H(\eta|\mathcal{N}(\mu,\Sigma))=-\mathcal{H}(\eta)-\mathbb{E}_{\eta}\left[\log\phi(z;\mu,\Sigma)\right],
\]
where
\[
\mathbb{E}_{\eta}\left[\log\phi(z;\mu,\Sigma)\right]=-\frac{d}{2}\log2\pi-\frac{1}{2}\log\det\left(\Sigma\right)-\frac{1}{2}\mathbb{E}_{\eta}[\left(z-\mu\right)^{T}\Sigma^{-1}\left(z-\mu\right)].
\]
Straightforward manipulations show
\begin{align*}
\mathbb{E}_{\eta}[\left(z-\mu\right)^{T}\Sigma^{-1}\left(z-\mu\right)] & =\mathbb{E}_{\eta}[z^{T}\Sigma^{-1}z]-\mu^{T}\Sigma^{-1}\mu\\
 & =Tr[\Sigma^{-1}\Sigma]+\mu^{T}\Sigma^{-1}\mu-\mu^{T}\Sigma^{-1}\mu\\
 & =d
\end{align*}
and hence
\[
H(\eta|\mathcal{N}(\mu,\Sigma))=-\mathcal{H}(\eta)+\frac{d}{2}\log2\pi+\frac{1}{2}\log\det\left(\Sigma\right)+\frac{d}{2}.
\]
Given $H(\eta|\mathcal{N}(\mu,\Sigma))\geq0$ we get the required result.
\end{proof}
\begin{corollary}\label{cov_ips_bound}
Let  $\eta^N_\infty\in\mathcal{P}_2(\mathbb{R}^{dN})$ and $\Sigma^N$ denote the covariance matrix of $\eta^N_\infty$. Then for the particle system it holds that
\begin{align}\label{eq:cov_ips_bound}
\mathcal{H}(\eta_{\infty}^{N})&\leq\frac{dN}{2}\left(\log2\pi+1\right)+\frac{1}{2}\log\det\left(\Sigma^{N}\right).
\end{align}
\end{corollary}
Next we provide a lower bound for $\log Z^N$.
\begin{proposition}\label{prop:logZ_lb} Let Assumptions \ref{ass:f}, \ref{ass:mirror1} hold w.r.t the Euclidean norm.  We have that $\nabla\mathcal{W}$ is $L_N$-Lipschitz and as a result
\[
\log Z^{N}\geq -\frac{2}{\sigma^{2}}\mathcal{W^{*}}+\frac{1}{2}\log\left(\frac{\sigma^{2}}{2L_{N}}\right)+\frac{dN}{2}\log2\pi .
\]
\end{proposition}
\begin{proof}
The first claim is trivial and one has $L_N= \frac{L}{\mu}+\left\Vert \mathcal{L}\right\Vert $ where $\left\Vert \mathcal{L}\right\Vert $ is any matrix norm on $\mathcal{L}$. Smoothness of $\nabla\mathcal{W}$ implies the following quadratic upper bounds:
\[
\mathcal{W}(\mathbf{z})\leq\mathcal{W}(\mathbf{z^*})+\frac{L_{N}}{2}\left\Vert \mathbf{z}-\mathbf{z^{*}}\right\Vert ^{2}.
\]
Then we have,
\begin{align*}
\log Z^{N} & =-\frac{2}{\sigma^{2}}\mathcal{W}(\mathbf{z^*})+\log\int\exp\left(-\frac{2}{\sigma^{2}}\left(\mathcal{W}(\mathbf{z})-\mathcal{W}(\mathbf{z^*})\right)\right)d\mathbf{z}\\
 & \geq-\frac{2}{\sigma^{2}}\mathcal{W}(\mathbf{z^*})+\log\int\exp\left(-\frac{L_{N}}{\sigma^{2}}\left\Vert \mathbf{z}-\mathbf{z^{*}}\right\Vert ^{2}\right)d\mathbf{z}\\
 & =-\frac{2}{\sigma^{2}}\mathcal{W}(\mathbf{z^*})+\frac{1}{2}\log\left(\frac{\sigma^{2}}{2L_{N}}\right)+\frac{dN}{2}\log2\pi,
\end{align*}
where in the last step we used the Gaussian integral $\int\exp\left(-\frac{1}{2}\mathbf{z}^{T}A\mathbf{z}+b^{T}\mathbf{z}\right)d\mathbf{z}=\sqrt{\frac{\left(2\pi\right)^{dN}}{\det A}}\exp\left(\frac{1}{2}b^{T}A^{-1}b\right)$
with $A=\frac{2L_N}{\sigma^{2}}I_{dN}$ and $b=A\mathbf{z^{*}}$.
\end{proof}
\begin{remark}
When non-Euclidean norms are used $L_N=\frac{L}{\mu}K+\left\Vert \mathcal{L}\right\Vert$, with $K$ arising from norm equivalence.
\end{remark}

\bibliographystyle{siam}
\bibliography{biblio}

\end{document}